\documentclass[a4paper,12pt]{amsart}

\usepackage{color}
\usepackage{amsfonts}
\usepackage{amsmath}
\usepackage{amscd}
\usepackage{amssymb,eucal}
\usepackage[headings]{fullpage}
\usepackage{algorithmic}
\usepackage[colorlinks=true,linkcolor=KITgreen,urlcolor=KITgreen,citecolor=blue,pdfpagelabels,plainpages=false]{hyperref}
\usepackage[ruled]{algorithm} 
\usepackage{todonotes}

\def\NN{\mathbb{N}}
\def\RR{\mathbb{R}}

\def\d{\mathrm{d}}

\newcommand{\ddiv}{\operatorname{div}}

\newcommand*{\V}{{\bf v}}
\newcommand*{\Sig}{{\boldsymbol \sigma}}
\newcommand*{\Etl}{{\boldsymbol \eta}_l}

\newcommand*{\tsigl}{\tau_{\Sig,l}}
\newcommand*{\Div}{\ensuremath{\mathrm{\,div\,}}}
\newcommand*{\f}{{\bf f}}

\newcommand*{\Id}{{\bf I}}
\newcommand*{\RM}{\ensuremath{\RR^{3\times 3}}}
\newcommand*{\RS}{\ensuremath{\RR^{3\times 3}_{\rm sym}}}

\newcommand*{\pr}{\text{\tiny P}}
\newcommand*{\sh}{\text{\tiny S}}

\newcommand*{\red}{\color{black}}

\allowdisplaybreaks

\newcommand{\trace}{\operatorname{tr}}

\def\vec#1{\ensuremath{\mathchoice
                     {\mbox{\boldmath$\displaystyle\mathbf{#1}$}}
                     {\mbox{\boldmath$\textstyle\mathbf{#1}$}}
                     {\mbox{\boldmath$\scriptstyle\mathbf{#1}$}}
                     {\mbox{\boldmath$\scriptscriptstyle\mathbf{#1}$}}}}

\definecolor{KITgreen}{cmyk}{    1,    0, 0.6,  0} 
\newtheorem{theorem}{Theorem}[section]
\newtheorem{lemma}[theorem]{Lemma}

\newtheorem{remark}[theorem]{Remark}

\begin{document} 
\title[Corrigendum to "Higher order differentiability for viscoelasticity"]{Corrigendum to "Inverse Problems for\\ abstract evolution equations~\rm{\bf II}:\\ \bfseries  higher order differentiability for viscoelasticity"}
\begin{abstract}
In  our paper \href{https://doi.org/10.1137/19M1269403}{[SIAM J.\ Appl.~Math.\ 79-6 (2019)]} we considered full waveform inversion 
(FWI) in the viscoelastic regime.  FWI entails the nonlinear  inverse problem of recovering  parameter functions of the 
viscoelastic wave equation from partial measurements of reflected wave fields. We have obtained explicit analytic expressions 
for the first and second order Fr\'echet derivatives and their adjoints (adjoint wave equations) of the underlying parameter-to-
solution map.

In the present manuscript we correct an error which occurred in a basic result and thus affected several other statements,
which we have also corrected
\end{abstract}
\author{Andreas Kirsch \and Andreas Rieder}
\thanks{Funded by the Deutsche Forschungsgemeinschaft (DFG, German Research Foundation)
– Project-ID 258734477 – SFB 1173.}
\address{Department of Mathematics, Karlsruhe Institute of Technology (KIT),
D-76128 Karls\-ruhe, Germany}
\email{andreas.kirsch@kit.edu, andreas.rieder@kit.edu}
\date{\today}
\keywords{full waveform seismic inversion, viscoelastic wave equation, adjoint state method, nonlinear inverse and ill-posed problem, higher order Fr\'echet derivative}
\subjclass[2000]{35F10, 35R30, 86A22}
\maketitle
\section{Introduction}
This corrigendum to our article \href{https://doi.org/10.1137/19M1269403}{[SIAM J.\ Appl.~Math.\ 79-6 (2019)]}  is necessary because of an  error in eqn.~\eqref{*} on page~\pageref{*} where we put the wrong right hand side in the original paper.\footnote{We thank Christian Rheinbay  for pointing this out to us.} 
The corrections concern the assertion and the proof of Theorem~\ref{th:second}, especially 
eqn.~\eqref{evolution3}, and the
proof of Theorem~\ref{th:Lipschitz}. As a consequence, Theorems~\ref{th:FWI_second}, \ref{th:adjoint2},  \ref{FWI''_adjoint1},
and \ref{2D_FWI''_adjoint1} had to be revised considerably.
For the sake of completeness and self-containedness we repeat here also those  parts of the original which have not been 
contaminated. For the general background and the motivation of our work see the Introduction of the original article.

We  apologize for any inconvenience our flaw may cause.
\section{Viscoelasticity} 
The viscoelastic wave equation in the
velocity stress formulation based on the generalized standard linear
solid (GSLS) rheology reads: In a Lipschitz domain $D\subset \RR^3$ we
determine the velocity field $ \vec{v} \colon [0,T]\times
D\to\RR^{3}$, the stress tensor $\vec\sigma\colon [0,T]\times
D\to\RR^{3\times 3}_\text{sym}$, and memory tensors
$\vec\eta_l\colon [0,T]\times D\to\RR^{3\times 3}_\text{sym}$,
\ $l=1,\dotsc,L$, from the first-order system
 \begin{subequations}
      \label{C2:elawave}
    \begin{align}
      \rho\,\partial_t \V &  = \Div\Sig+\f & & \text{in }[0,T]\times D,
        \label{C2:elawave1}\\
        \partial_t \Sig &
        = C\big((1+L\tau_\sh)\mu_{0},(1+L\tau_\pr)\pi_{0}  \big)\,\varepsilon(\V)
        +\sum_{l=1}^L \Etl &  &\text{in }[0,T]\times D,\label{C2:elawave2}\\
          -\tsigl\partial_t \Etl &
          = C\big(\tau_\sh\mu_0, \tau_\pr\pi_0\big)\,\varepsilon(\V) +\Etl,\qquad\qquad l=1,\dots,L,&
          &\text{in }[0,T]\times D.
            \label{C2:elawave3}
    \end{align}
    \end{subequations}
Here, $\f$ denotes the external volume force density and $\rho$ is the mass density.
The linear maps $C(m,p)$ for $m,p\in\RR$ are defined as
    \begin{equation} \label{C}
      C(m,p)\colon\RM \to \RM,\qquad
      C(m,p){\vec M}
      =2m\vec M + (p-2m) \trace(\vec M)\Id,
    \end{equation}
{\red with  $\Id\in \RM $ being the identity matrix and $\trace(\vec M)$ denotes the trace of  $\vec M\in \RM $.} Further, 
$$
    \vec \varepsilon( \vec{v}) =
    \frac{1}{2}\big[(\nabla_x \vec{v})^\top+\nabla_x  \vec{v}\big]
    $$
is the linearized strain rate.  In formulation \eqref{C2:elawave} two
independent GSLS are used to describe the propagation of pressure and shear waves (P- and
S-waves). The parameters $ \mu_{0}$ and $\pi_0$
denote the relaxed
P- and S-wave modulus, respectively. Further,   $\tau_{\pr}$ and $\tau_{\sh}$ are scaling factors for
the relaxed moduli.  They have been introduced for the first time by
\cite{Blanch1995} and are now widely used to quantify attenuation
and phase velocity dispersion in viscoelastic media, see e.g. \cite{Fichtner2011,Yang:2016a}.

Wave propagation in viscoelastic media is frequency-dependent over a bounded frequency band with center frequency $\omega_0$. Within this band the Q-factor, which is the rate of the full energy over the dissipated energy, remains nearly constant. 
This fact is used to determine the stress relaxation times $\tsigl>0$  by a least-squares approach~\cite{bohlen98,bohlen2002} where
up to $L=5$ relaxation mechanisms may be required. Now we obtain the following frequency-dependent 
phase velocities of P- and S-waves:
      \begin{equation}
        \label{V_sp}
        v^{2}_{\pr} =\frac{\pi_{0}}\rho(1+ \tau_{\pr} \alpha ) 
        \quad\text{and}\quad  
        v^{2}_{\sh} =\frac{\mu_{0}}\rho(1+ \tau_{\sh} \alpha)\quad \text{with } \alpha =
         \alpha(\omega_0) =\sum_{l=1}^{L} \frac{\omega_0^{2}\tsigl^{2}}{1+\omega_0^{2}\tsigl^{2}}.
      \end{equation} 
Full waveform inversion (FWI) in seismic imaging entails the inverse problem of reconstructing  the five spatially dependent parameters 
$ (\rho,v_{\sh}, \tau_{\sh}, v_{\pr},\tau_{\pr})$ from wave field measurements.

Using the transformation 
$$
\begin{pmatrix}
\vec v\\
\vec \sigma_0\\
\vec \sigma_1\\
\vdots\\
\vec \sigma_L
\end{pmatrix}:=
\begin{pmatrix}
\vec v\\
\vec \sigma+\sum_{l=1}^{L}\tsigl\vec\eta_l\\
-\tau_{\Sig,1} \vec\eta_1\\
\vdots\\
-\tau_{\Sig,L} \vec\eta_1
\end{pmatrix}
$$
discovered and explored by  Zeltmann~\cite{zeltmann_phd} we reformulate \eqref{C2:elawave} equivalently into
\begin{subequations}
      \label{C2:elawave_trans}
    \begin{align}
      \partial_t \V &  = \frac{1}{\rho}\Div \Big(\sum_{l=0}^{L}\vec\sigma_l\Big)+\frac{1}{\rho}\,\f & & \text{in }[0,T]\times D,
        \label{C2:elawave1_trans}\\
        \partial_t \Sig_0 &
        = C\big(\mu_{0},\pi_{0}  \big)\,\varepsilon(\V)
       &  &\text{in }[0,T]\times D
          ,\label{C2:elawave2_trans}\\
          \partial_t \Sig_l &
          = C\big(\tau_\sh\mu_0, \tau_\pr\pi_0\big)\,\varepsilon(\V) -\frac{1}{\tsigl}\,\Sig_l,\qquad l=1,\dots,L,&
          &\text{in }[0,T]\times D.
            \label{C2:elawave3_trans}
    \end{align}
    \end{subequations}
Let $X=L^2(D,\RR^3)\times L^2(D,\RS)^{1+L}$. For suitable\footnote{A rigorous mathematical formulation will be given in 
Section~\ref{sec:application} below.} $w=(\vec w,\vec\psi_0,\dotsc,\vec \psi_L)\in X$  we define
the operators $A$, $B$, and $Q$ mapping into $X$ by
\begin{equation}\label{the_operators}
Aw=-\begin{pmatrix}
\Div \big(\sum_{l=0}^{L}\vec\psi_l\big) \\[1mm] \varepsilon(\vec w)\\[1mm] \vdots\\[1mm] \varepsilon(\vec w)\end{pmatrix},\ 
B^{-1}w=\begin{pmatrix}
\frac{1}{\rho}\,\vec w\\[1mm]
C\big(\mu_{0},\pi_{0}  \big)\vec\psi_0\\[1mm]
C\big(\tau_\sh\mu_0, \tau_\pr\pi_0\big)\vec\psi_1\\[1mm]
\vdots\\[1mm]
C\big(\tau_\sh\mu_0, \tau_\pr\pi_0\big)\vec\psi_L
\end{pmatrix},\ 
Qw=\begin{pmatrix}
\vec 0\\[1mm]
\vec 0\\[1mm]
\frac{1}{\tau_{\Sig,1}}\,\vec\psi_1\\[1mm]
\vdots\\[1mm]
\frac{1}{\tau_{\Sig,L}}\,\vec\psi_L
\end{pmatrix}.
\end{equation}
With these operators the system \eqref{C2:elawave_trans} can be rewritten as
$$
Bu'(t) +Au(t) + BQu(t) =f(t)
$$
where $u=(\V,\Sig_0,\dotsc,\Sig_L)$ and $f=(\vec f,\vec 0,\dotsc,\vec 0)$. 

Please note: The five parameters  to  be reconstructed by FWI enter only the operator~$B$ via, see \eqref{V_sp},
\begin{equation}\label{V_p2}
 \pi_0 =\frac{ \rho \,v^{2}_{\pr}}{1+ \tau_{\pr} \alpha }
        \quad\text{and}\quad  
     \mu_0 =\frac{ \rho \,v^{2}_{\sh}}{1+ \tau_{\sh} \alpha }.
\end{equation}
\section{Abstract framework}\label{sec:abstract}
We consider an abstract evolution equation in a Hilbert space $X$ of the form 
\begin{equation}\label{evolution}
Bu'(t) +Au(t) + BQu(t) =f(t), \quad t\in \,] 0,T[,\quad u(0)=u_0,
\end{equation}
under the following general  hypotheses: $T>0$, $u_0\in X$,
\begin{itemize}
\item[]
$B$ belongs to the Banach space $\mathcal{L}^*(X)=\{ J\in \mathcal{L}(X): J^*=J\}$ and satisfies $\langle Bx,x\rangle_X= 
\langle x,Bx\rangle_X\ge \beta \Vert x\Vert_X^2$ for some $\beta>0$ and for all $x\in X$, \vspace{2mm}
\item[] $A\colon \mathsf{D}(A)\subset X\to X$ is a maximal monotone operator: $\langle Ax,x\rangle_X\ge 0$ for all $x\in   \mathsf{D}(A)$ and
$I+A\colon  \mathsf{D}(A) \to X$ is onto ($I$ is the identity),\vspace{2mm}
\item[]
 $Q\in \mathcal{L}(X)$, and $f\in L^1([0,T],X)$.
\end{itemize}
Later we will show that the three operators from \eqref{the_operators} are well defined and 
satisfy  our general hypotheses in a precise mathematical setting.  

In \cite{Kirsch-Rieder16} we explored \eqref{evolution} with $Q=0$.
Existence and regularity results of  this paper  apply correspondingly. 
Let us be more precise: equation \eqref{evolution} can be transformed equivalently in
$$
u'(t) +(B^{-1}A + Q)u(t) =B^{-1}f(t), \quad t\in \,] 0,T[,\quad u(0)=u_0,
$$
where $B^{-1}A $ with $\mathsf{D}( B^{-1}A)=\mathsf{D}(A)$ generates a contraction semigroup on 
$(X,\langle\cdot,\cdot\rangle_B)$ with weighted inner product $\langle\cdot,\cdot\rangle_B:=\langle B\cdot,\cdot\rangle_X$ 
where the  induced norm $\Vert \cdot\Vert_B$ is equivalent to the original norm on $X$. Further, $B^{-1}A + Q$ is the infinitesimal 
generator of  a $C_0$-semigroup $\{S(t)\}_{t\ge 0}$ with
$$
\Vert S(t)\Vert_B\le \exp (\Vert Q\Vert_B t),
$$
see, e.g., Theorem~3.1.1  of \cite{pazy_semi}. Thus, \eqref{evolution} has a unique mild/weak solution in $\mathcal{C}([0,T],X)$  
given by
$$
u(t)=S(t)u_0+\int_0^t S(t-s) B^{-1}f(s)\,\d s.
$$
{\red
On the basis of the above comments, both Theorems~2.4 and 2.6 of \cite{Kirsch-Rieder16}} carry over to \eqref{evolution} when  we replace $f$  by $B^{-1}f$ and compensate  the use of $\Vert\cdot\Vert_X$ by an additional constant depending on $\Vert B\Vert$, $\Vert B^{-1}\Vert$, $\Vert Q\Vert$, and~$T$. Thus, we have the continuous dependence of $u$ on the data:
\begin{equation}\label{cont_dep_on_data}
\Vert u\Vert_{\mathcal{C}([0,T],X)}\lesssim \Vert u_0\Vert_X+\Vert f\Vert_{ L^1([0,T],X)}\footnote{$A\lesssim B$ indicates the existence of a generic constant $c>0$ such that $A\le c \,B$.}
\end{equation}
{\red
as well as the  following regularity result which has been 
shown in \cite[Theorem~2.6]{Kirsch-Rieder16} for $Q=0$ under more general assumptions on $f$ and~$u_0$.
\begin{theorem} \label{t-regul_2}
For some $k\in\NN$, 
let $f\in W^{k,1}([0,T],X)$ with $f^{(\ell)}(0)=0$, $\ell=0,\ldots,k-1$ (note that 
$f^{(\ell)}$ is continuous). Let $B\in\mathsf{D}(F)$ and let 
$u$ be the unique mild solution of 
\eqref{evolution} with $u_0=0$.
Then $u\in \mathcal{C}^k([0,T],X)\cap \mathcal{C}^{k-1}([0,T],\mathsf{D}(A))$ and
\begin{equation} \label{reg}
\Vert u\Vert_{\mathcal{C}^k([0,T],X)}\lesssim \Vert f\Vert_{W^{k,1}([0,T],X)}
\end{equation}
where the constant depends on $T$, $Q$, $\beta_-$, and $\beta_+$.
\end{theorem} }
\subsection{Abstract parameter-to-solution map}
We define the following parameter-to-solution map related to \eqref{evolution}: 
\begin{equation}\label{F}
F\colon \mathsf{D}(F)\subset \mathcal{L}^*(X)\to  \mathcal{C}([0,T],X),\quad B\mapsto u,
\end{equation}
where 
$$
\mathsf{D}(F)=\{B\in\mathcal{L}^*(X): \beta_-\Vert x\Vert_X^2\le \langle Bx,x\rangle_X\le\beta_+ \Vert x\Vert_X^2\} 
$$
for given $0< \beta_-<\beta_+<\infty$. 

Transferring the techniques of proof of \cite[Theorem~3.6]{Kirsch-Rieder16} straightforwardly to $F$ yields the following result.
\begin{theorem}\label{th:first}
Let $T>0$, $f\in W^{1,1}\big([0,T],X\big)$, and $u_0\in\mathsf{D}(A)$. Then, the mild solution of \eqref{evolution} is a classical solution, i.e., 
${u}\in \mathcal{C}^1\bigl([0,T],X\bigr)\cap \mathcal{C}\bigl([0,T],\mathsf{D}(A)\bigr)$, and 
$F$ is Fr\'{e}chet dif\-fer\-en\-ti\-able at
${B}\in\operatorname{int}(\mathsf{D}(F))$ with $F^\prime({B})H=\overline{u}$, $H\in  \mathcal{L}^*(X)$,  where $\overline{u}\in
\mathcal{C}\bigl([0,T],X\bigr)$ is the mild  solution of
\begin{equation} \label{evolution2}
{B}\overline{u}^\prime(t)+A\overline{u}(t)+BQ\overline{u}(t)=  -H({u}^\prime(t)+Qu(t)),\ t\in\,[0,T],\quad\overline{u}(0)=0.
\end{equation}
\end{theorem}
The representation of the adjoint of the Fr\'echet derivative carries over as well, see \cite[Theorem~3.8]{Kirsch-Rieder16}.
\begin{theorem}\label{th:adjoint}
Under the notation and assumptions of Theorem~{\rm\ref{th:first}} we have
$$ \big[F^\prime({B})^\ast g\big]H= \int_0^T\big\langle H({u}^\prime(t)+Qu(t)),w(t)\big\rangle_X\d t,\quad g\in L^2([0,T],X),\
H\in\mathcal{L}^*(X), $$
where $w\in \mathcal{C}([0,T],X)$ is the mild solution of the {\red adjoint} evolution equation
\begin{equation}\label{evolution_back}
{B}w^\prime(t) - A^\ast w(t) -Q^*Bw(t)=g(t),\  t\in\,[0,T],\quad w(T)=0.
\end{equation}
\end{theorem} 
\begin{remark}{\red Setting $\widetilde{w}(t)=w(T-t)$ and $\widetilde{g}(t)=g(T-t)$ we rewrite \eqref{evolution_back}
as initial value problem 
$$
B\widetilde{w}^\prime(t) + A^\ast \widetilde{w}(t) +Q^*B\widetilde{w}(t)=-\widetilde{g}(t),\  t\in\,[0,T],\quad \widetilde{w}(0)=0,
$$
which is of the same structure as our original equation \eqref{evolution} since $A^*$ is maximal monotone as well. 
Further, in our concrete setting of the viscoelastic wave equation we have $A^*=-A$ (see the next section) so that basically the same numerical solver can be used for the state and the adjoint state equation. 

This remark applies also to the situation of Theorem~{\rm\ref{th:adjoint2}} below.
}\end{remark}
Next we investigate  second order differentiability of $F$.
\begin{theorem}\label{th:second}
Let $f\in W^{3,1}([0,T],X)$,  $u_0=0$, and  $f(0)=f'(0)=f''(0)=0$.
Then,  $F$ is twice Fr\'echet differentiable  at 
$B\in  \operatorname{int}(\mathsf{D}(F))$ with $F''(B)[H_1,H_2]=\overline{\overline{u}}$, $H_i\in \mathcal{L}^*(X)$, $i=1,2$, where $\overline{\overline{u}}\in  \mathcal{C}([0,T],X)$  is the mild (in fact the classical) solution of 
\begin{equation}\label{evolution3}
B\overline{\overline{u}}\,'(t)+(A+BQ)\overline{\overline{u}}(t)=-H_1({u}_2'(t)+Q{u}_2(t))
-H_2({u}_1'(t)+Q{u}_1(t)), \quad \overline{\overline{u}}(0)=0.
\end{equation}
Here, $u_i\in \mathcal{C}^2([0,T],X)\cap \mathcal{C}^1([0,T],\mathsf{D}(A))$, $i=1,2$, is the classical solution of 
\eqref{evolution2} with $H$ replaced by $H_i$: 
\begin{equation}\label{evolution4}
B{u}_i'(t)+(A+BQ){u}_i(t)=-H_i(u'(t)+Qu(t)), \quad {u}_i(0)=0, \quad i=1,2.
\end{equation}
Further, $u \in \mathcal{C}^3([0,T],X)\cap \mathcal{C}^2([0,T],\mathsf{D}(A))$ solves \eqref{evolution}.
\end{theorem}
\begin{proof}
We need to show that
$$
\sup_{H_2\in \mathcal{L}^*(X)}\frac{\Vert F'(B+H_1)H_2-F'(B)H_2-F''(B)[H_1,H_2]\Vert_{\mathcal{C}([0,T],X)}}{\Vert H_1\Vert_{\mathcal{L}(X)}\Vert H_2\Vert_{\mathcal{L}(X)} }\xrightarrow{\,H_1\to 0\,} 0.
$$
Set $u_+:=F(B+H_1)$ and $\widetilde{u}:=F'(B+H_1)H_2$ which is well defined  for $H_1$ sufficiently small. We have 
\begin{align*}  
B{u}_i^\prime  +  (A+BQ)u_i & \,=  -H_i(u^\prime+Qu),\ i=1,2, \\
(B+H_1)\widetilde{u}\,^\prime  +  \bigl(A+(B+H_1)Q\bigr)\widetilde{u} & \,=  -H_2(u^\prime_++Qu_+), \tag{$\ast$}
\label{*}\\
B\overline{\overline{u}}\,^\prime  +  (A+BQ)\overline{\overline{u}} & \,= 
-H_1({u}_2'+Q{u}_2)-H_2({u}_1'+Q{u}_1).
\end{align*}
Then, $\widetilde{u}-u_2$ and $v:=\widetilde{u}-{u}_2-\overline{\overline{u}}$ satisfy
\begin{equation} \label{eq2}
B(\widetilde{u}-{u}_2)^\prime + (A+BQ)(\widetilde{u}-{u}_2)= -H_1(\widetilde{u}\,^\prime+Q\widetilde{u})
-H_2\big( (u_+-u)'+Q(u_+-u)\big)
\end{equation}
and
\begin{multline*} 
Bv^\prime + (A+BQ)v = -H_1\bigl[(\widetilde{u}-{u}_2)^\prime+Q(\widetilde{u}-{u}_2)\bigr] \\
-H_2\big( (u_+-u-u_1)'+Q(u_+-u-u_1)\big),
\end{multline*}
respectively, with homogeneous initial conditions. Using the continuous dependence of $v$ on the right hand side, see  \eqref{cont_dep_on_data},  we get
\begin{equation}\label{eq:aux1}
\Vert v\Vert_{\mathcal{C}([0,T],X)}
\lesssim \Vert H_1\Vert_{ \mathcal{L}(X)}\,\Vert \widetilde{u}-{u}_2\Vert_{\mathcal{C}^1([0,T],X)}
+\Vert H_2\Vert_{ \mathcal{L}(X)}\,\Vert u_+-u-u_1\Vert_{\mathcal{C}^1([0,T],X)}.
\end{equation}
Now we apply the regularity estimate \eqref{reg} repeatedly  for $k=1$ to $\widetilde{u}-{u}_2$  in \eqref{eq2}, 
then for $k=2$ to $\widetilde{u}$ as well as to $u_+-u$, which satisfies $B(u_+-u)'+(A+BQ)(u_+-u)=H_1 u_+'$, 
and finally for $k=3$ to $u$:
\begin{align*}
\Vert \widetilde{u}-{u}_2\Vert_{\mathcal{C}^1([0,T],X)} & \lesssim\Vert H_1\Vert_{\mathcal{L}(X)}\Vert\widetilde{u}
\Vert_{\mathcal{C}^2([0,T],X)} + \Vert H_2\Vert_{\mathcal{L}(X)}\Vert u_+-u\Vert_{\mathcal{C}^2([0,T],X)}\\[1mm]
&\lesssim \Vert H_1\Vert_{\mathcal{L}(X)}\Vert H_2\Vert_{\mathcal{L}(X)} \Vert u\Vert_{\mathcal{C}^3([0,T],X)} 
 \lesssim\Vert H_1\Vert_{\mathcal{L}(X)}\Vert H_2\Vert_{\mathcal{L}(X)}\Vert f\Vert_{W^{3,1}([0,T],X)}. 
\end{align*}
Let $d:=u_+-u-u_1$. Then, $Bd'+(A+BQ)d=-H_1((u_+-u)'+Q(u_+-u))$, $d(0)=0$. Again,  \eqref{reg}  gives
\begin{align*}
\Vert d\Vert_{\mathcal{C}^1([0,T],X)} &\lesssim \Vert H_1\Vert_{\mathcal{L}(X)} \Vert u_+-u\Vert_{\mathcal{C}^2([0,T],X)}
\lesssim  \Vert H_1\Vert_{\mathcal{L}(X)}^2 \Vert u\Vert_{\mathcal{C}^3([0,T],X)} \\[1mm]
&\lesssim  \Vert H_1\Vert_{\mathcal{L}(X)}^2  \Vert f\Vert_{W^{3,1}([0,T],X)}.
\end{align*}
Substituting the latter two bounds into \eqref{eq:aux1}  yields
$$
\frac{1}{\Vert H_1\Vert_{\mathcal{L}(X)}}\sup_{H_2\in \mathcal{L}(X)} \frac{\Vert \widetilde{u}-\overline{u}- \overline{\overline{u}}\Vert_{\mathcal{C}([0,T],X)}}{\Vert H_2\Vert_{\mathcal{L}(X)}} \lesssim 
\Vert H_1\Vert_{\mathcal{L}(X)} \Vert f\Vert_{W^{3,1}([0,T],X)}
$$
which finishes the proof. 
\end{proof}
\begin{remark}
In  seismic exploration, where \eqref{evolution} is the viscoacoustic or viscoelastic wave equation, we can 
assume the  environment to be at rest before firing the source.  In other words, the assumptions on $u_0$ and $f$ from the 
above theorem are justified.

The symmetry of second order Fr\'echet derivative, see, e.g., \cite[(8.12.2)]{dieudonne}, is clearly visible in \eqref{evolution3}.
\end{remark}
The mindful reader might have noticed an unbalanced increase of the smoothness assumptions on $f$ and $u_0$ from 
Theorem~{\rm\ref{th:first}}   ($f\in W^{1,1}$) to Theorem~{\rm\ref{th:second}} ($f\in W^{3,1}$)  compared to the increase of smoothness of $F$: two additional 
differentiation orders for $f$ gain only  one  order for $F$. This is because in \eqref{eq:aux1} we need  convergence of $\Vert 
\widetilde{u}-\overline{u}\Vert_{\mathcal{C}^1([0,T],X)}\to 0$ as $H_1\to 0$ {\em uniformly} in $H_2$. 
At least we get  $F\in\mathcal{C}^{2,1}$, that is,  $F''$ is  uniformly Lipschitz continuous.
\begin{theorem}\label{th:Lipschitz}
Under  the assumptions of Theorem~{\rm\ref{th:second}}  we have that\footnote{$\mathcal{L}^2(V, W)$ 
denotes the space of bounded bilinear mappings from $V$ to $W$.}
$$
\Vert F''(B)-F''(\widetilde{B})\Vert_{\mathcal{L}^2(\mathcal{L}^*(X),\mathcal{C}([0,T],X))} \lesssim
 \Vert B-\widetilde{B}\Vert_{\mathcal{L}(X)}
$$
uniformly in $\operatorname{int}(\mathsf{D}(F))$. The constant in the above estimate only depends on $\beta_-$, 
$\beta_+$, $T$, $Q$, and $f$.
\end{theorem}
\begin{proof}
Set $\delta B:=\widetilde{B}-B$.
For $H_i\in\mathcal{L}^*(X)$, $i=1,2$, we estimate
$\Vert \overline{\overline{u}}-\overline{\overline{v}}\Vert_{\mathcal{C}([0,T],X)}$ where 
$\overline{\overline{v}}=F''(B+\delta B)[H_1,H_2]$, $\overline{\overline{u}}=F''( B)[H_1,H_2]$.  From \eqref{evolution3} we get 
\begin{multline*}
B(\overline{\overline{v}}\,'-\overline{\overline{u}}\,')+(A+BQ)(\overline{\overline{v}}-\overline{\overline{u}})
=-H_1({v}_2'-{u}_2'+Q({v}_2-{u}_2)) \\
-H_2({v}_1'-{u}_1'+Q({v}_1-{u}_1))
-\delta B( \overline{\overline{v}}\,' +Q
\overline{\overline{v}})
\end{multline*}
where  ${u}_i$ is the solution of \eqref{evolution4} and ${v}_i$ solves \eqref{evolution4} with $B$ replaced by 
$B+\delta B$ and $u$ by $v$, the latter being the solution of \eqref{evolution}  with $B+\delta B$
instead of $B$ and $v(0)=0$.
As before,  by the continuous dependence on the right hand side,
\begin{multline} \label{eq:aux2}
\Vert \overline{\overline{v}}-\overline{\overline{u}}\Vert_{\mathcal{C}([0,T],X)} \lesssim \Vert H_1\Vert_{\mathcal{L}(X)}
\Vert {v}_2-{u}_2\Vert_{\mathcal{C}^1([0,T],X)}+
\Vert H_2\Vert_{\mathcal{L}(X)}
\Vert {v}_1-{u}_1\Vert_{\mathcal{C}^1([0,T],X)}\\
+\Vert \delta B\Vert_{\mathcal{L}(X)}
\Vert \overline{\overline{v}}\Vert_{\mathcal{C}^1([0,T],X)}
\end{multline}
where the involved constant only depends on $\beta_-$, $\beta_+$, $T$, and $Q$. All constants in this proof, which are not explicitly given, only depend on these four quantities.

Further, by applying \eqref{reg} again repeatedly for $k=1$, $k=2$, and $k=3$, we obtain
\begin{align}\label{eq:aux4}
\Vert \overline{\overline{v}}\Vert_{\mathcal{C}^1([0,T],X)}
&\lesssim \Vert H_1\Vert_{\mathcal{L}(X)} \Vert{v}_2\Vert_{\mathcal{C}^2([0,T],X)}
+ \Vert H_2\Vert_{\mathcal{L}(X)} \Vert{v}_1\Vert_{\mathcal{C}^2([0,T],X)} \\[1mm]
&\lesssim \Vert H_1\Vert_{\mathcal{L}(X)} \Vert H_2\Vert_{\mathcal{L}(X)}  \Vert v\Vert_{\mathcal{C}^3([0,T],X)}
\lesssim \Vert H_1\Vert_{\mathcal{L}(X)} \Vert H_2\Vert_{\mathcal{L}(X)}\Vert f\Vert_{W^{3,1}([0,T],X)}. \nonumber
\end{align}
In view of \eqref{eq:aux2} it remains to investigate $\Vert {v}_i-{u}_i\Vert_{\mathcal{C}^1([0,T],X)}$, $i=1,2$.
We can use the same approach as above: Set $\overline{d}={v}_i-{u}_i$ and $d=v-u$.  Then, $\overline{d}(0)=0$ and
$$
B\overline{d}\,'+(A+BQ)\overline{d}=-H_i(d'+Qd)-\delta B ( {v}_i'+Q{v_i}).
$$
By \eqref{reg} as well as the second and third estimate from \eqref{eq:aux4},
\begin{align*}
\Vert \overline{d}\Vert_{\mathcal{C}^1([0,T],X)}
&\lesssim  \Vert H_i\Vert_{\mathcal{L}(X)}\big(\Vert d\Vert_{\mathcal{C}^2([0,T],X)}+\Vert \delta B\Vert_{\mathcal{L}(X)} \Vert f\Vert_{W^{3,1}([0,T],X)} \big).
\end{align*}
We are left with estimating $\Vert d\Vert_{\mathcal{C}^2([0,T],X)}$. Note that
$$
Bd'+(A+BQ)d=-\delta B(v'+Qv)
$$
and  \eqref{reg}  delivers 
$$
\Vert d\Vert_{\mathcal{C}^2([0,T],X)}\lesssim \Vert \delta B\Vert_{\mathcal{L}(X)}
\Vert v\Vert_{\mathcal{C}^3([0,T],X)} \lesssim \Vert \delta B\Vert_{\mathcal{L}(X)}
\Vert f\Vert_{W^{3,1}([0,T],X)} .
$$
So we found that
$$
\Vert {v}_i-{u}_i\Vert_{\mathcal{C}^1([0,T],X)} \lesssim 
\Vert H_i\Vert_{\mathcal{L}(X)}\Vert \delta B\Vert_{\mathcal{L}(X)} \Vert f\Vert_{W^{3,1}([0,T],X)} .
$$
Plugging this bound together with \eqref{eq:aux4} into  \eqref{eq:aux2} results in
$$
\sup_{H_1,H_2\in{\mathcal{L}^*(X)}}\frac{\Vert \overline{\overline{v}}-\overline{\overline{u}}\Vert_{\mathcal{C}([0,T],X)}}{
\Vert H_1\Vert_{\mathcal{L}(X)} \Vert H_2\Vert_{\mathcal{L}(X)}} \lesssim \Vert f\Vert_{W^{3,1}([0,T],X)}  
\Vert \delta B\Vert_{\mathcal{L}(X)}
$$
and we are done.
\end{proof}
\subsection{\color{black}Local ill-posedness} 
{\red
We recall briefly the concept of local ill-posedness from~\cite{hofmann_preprint97}: 
Let $\Psi\colon  \mathsf{D}(\Psi) \subset \mathcal{X}\to\mathcal{Y}$ be a mapping between infinite dimensional normed spaces.
Then, the equation $\Psi(\cdot)=y$  is {\em locally ill-posed} at $x^+\in \mathsf{D}(\Psi)$ if
in any neighborhood $U$ of $x^+$ there exists a sequence $\{\xi_k\}\subset U\cap \mathsf{D}(\Psi)$ with 
$\lim_{k\to\infty} \Vert \Psi(\xi_k)-y\Vert_\mathcal{Y}=0$ but  $\{\xi_k\}$ does not converge to~$x^+$ in $\mathcal{X}$.
}

Here, we consider \eqref{F} as a mapping with the larger image space $L^2([0,T],X)$. 
Theorem~4.1  of \cite{Kirsch-Rieder16} applies directly to \eqref{evolution} and \eqref{F}. The proof only needs a slight and obvious modification.
\begin{theorem} \label{t-illposed_4}
Let $u$ be the classical solution of \eqref{evolution} for
$u_0\in\mathsf{D}(A)$ and $f\in W^{1,1}([0,T],X)$. Then the equation $F(B)=u$ is locally ill-posed at any 
$\widehat{B}\in\mathsf{D}(F)$ satisfying $F(\widehat{B})=u$ if for any $r\in(0,1]$ there exists $\widehat{r}\in(0,r)$ and a sequence of bounded, symmetric 
and monotone operators $E_k\colon X\to X$ such that $\widehat{B}+E_k\in\mathsf{D}(F)$, $\widehat{r}\leq\Vert E_k\Vert_{\mathcal{L}(X)}\leq r$ for all 
$k\in\NN$, and $\lim_{k\to\infty}E_k v=0$ for all $v\in X$.
\end{theorem}
\section{Application to the viscoelastic wave equation} \label{sec:application}
We apply the abstract results to the viscoelastic wave equation in the formulation \eqref{C2:elawave_trans}. 
The underlying Hilbert space is 
$$
X=L^2(D,\RR^3)\times L^2(D,\RS)^{1+L}
$$
with inner product 
$$ 
\big\langle (\mathbf{v}, \Sig_0,\dotsc, \Sig_L), (\mathbf{w}, \vec\psi_0,\dotsc, \vec\psi_l)\big\rangle_X=
\int_D\Big(\mathbf{v}\cdot\mathbf{w}+\sum_{l=0}^L\boldsymbol\sigma_l:\boldsymbol\psi_l\Big)\,\d x
$$
where the colon indicates the Frobenius inner product on $\RR^{3\times 3}$.

To define the domain $\mathsf{D}(A)$ of $A$ \eqref{the_operators} we split
the boundary $\partial D$ of the bounded 
Lipschitz domain $D$  into disjoint parts $\partial D=\partial D_D\,\dot{\cup}\,\partial D_N$. Let $\mathbf{n}$ be the outer normal vector on $\partial D_N$. Then,
$$
\mathsf{D}(A)=\Big\{ (\vec w,\vec\psi_0,\dotsc \vec\psi_L)\in H^1_D\times H(\Div)^{1+L}: \sum_{l=0}^L\vec\psi_l\vec n=0
\text{ on } \partial D_N\Big\}
$$
with 
$H^1_D=\{\mathbf{v}\in H^1(D,\RR^3):\mathbf{v}=0\text{ on }\partial D_D\}$ and 
$H(\Div)=\bigl\{\boldsymbol\sigma\in
L^2\bigl(D,\RR^{3\times 3}_\text{sym}\bigr):\ddiv\boldsymbol\sigma_{*,j}\in L^2(D),\ j=1,2,3\bigr\}$.\footnote{The traces
$\boldsymbol\sigma_{*,j}\cdot\mathbf{n}$ exist in a suitable space, see, e.g., \cite{monk2003}.}
\begin{lemma}
The operator $A$ as defined in \eqref{the_operators} with $\mathsf{D}(A)\subset X$ from above is maximal monotone.
\end{lemma}
\begin{proof}
Since 
$$
\big\langle A (\mathbf{v}, \Sig_0,\dotsc, \Sig_L), (\mathbf{w}, \vec\psi_0,\dotsc, \vec\psi_l)\big\rangle_X=
\int_D\Big[\Div \Big(\sum_{l=0}^L \Sig_l\Big)\cdot\mathbf{w}+\vec\varepsilon(\V):\Big(\sum_{l=0}^L\boldsymbol\psi_l\Big)\Big]\,\d x
$$
we can proceed exactly as in the proof of Lemma~6.1 from \cite{Kirsch-Rieder16} to show skew-symmetry of $A$. Hence,
$\langle Aw,w\rangle_X=0$ for all $w\in  \mathsf{D}(A)$.

Next we show that $I+A$ is onto adapting arguments of  \cite{Kirsch-Rieder16}. We will be brief therefore. For $(\vec f,\vec g_0,\dotsc,\vec g_L)\in X$ we need to find $(\V,\Sig_0,\dotsc,\Sig_L)\in  \mathsf{D}(A)$ satisfying
$$
\V - \Div \Big(\sum_{l=0}^L \Sig_l\Big) = \vec f,\qquad \Sig_l -\vec\varepsilon(\V)=\vec g_l,\quad l=0,\dotsc,L.
$$
We multiply the  equation on the left  by a
$\mathbf{w}\in H^1_D$, integrate over $D$ and use the divergence theorem to get
$$
\int_D \Big(\V\cdot\vec w+ \Big(\sum_{l=0}^L \Sig_l\Big) :\nabla \vec w\Big)\d x =\int_D \vec f\cdot \vec w\,\d x.
$$
Now we sum up the $L+1$ equations on the right, use the relation $\vec\varepsilon(\V):\Sig=\nabla \V :\Sig$ {\red for $\Sig\in \RR^{3\times 3}_\text{sym}$}, and arrive at
$$
\int_D \big(\V\cdot\vec w+ (L+1)\vec\varepsilon(\V) :\vec\varepsilon( \vec w)\big)\d x =\int_D \Big(\vec f\cdot \vec w
-\sum_{l=0}^L  \vec g_l :\nabla\vec w\Big)\d x\quad\text{for all }\vec w\in H^1_D.
$$
This is a standard variational problem (cf.\ displacement ansatz in elasticity) admitting a unique solution $\mathbf{v}\in H^1_D$. 

Set $\Sig_l=\vec g_l+\vec\varepsilon(\V)$ and follow \cite{Kirsch-Rieder16} to verify $(\V,\Sig_0,\dotsc,\Sig_L)\in  \mathsf{D}(A)$.
\end{proof}
Next we show that $B\in \mathcal{L}(X)$ from \eqref{the_operators} is well defined with the required properties.
As in \cite{Kirsch-Rieder16} we consider $C$ of \eqref{C}  as a mapping from $\mathsf{D}(C)=\bigl\{(m,p)\in\RR^2: \underline{\mathrm{m}}\le 
m\le\overline{\mathrm{m}},\ \underline{\mathrm{p}}\le p\le \overline{\mathrm{p}}\bigr\}$ into $\operatorname{Aut}(\RR^{3\times 3}_\text{sym})$\footnote{\red This is the space of linear maps from $\RR^{3\times 3}_\text{sym}$ into itself (space of 
automorphisms).} with constants $0<\underline{\mathrm{m}}<\overline{\mathrm{m}}$ and
$0<\underline{\mathrm{p}}<\overline{\mathrm{p}}$ such that $3\underline{\mathrm{p}}>4 \overline{\mathrm{m}}$.\footnote{Note that in \cite{Kirsch-Rieder16}  and \cite{zeltmann_phd}   different $C$'s are used.} For $(m,p)\in\mathsf{D}(C)$,
\begin{equation} \label{C^-1}
\widetilde{C}(m,p) :=C(m,p)^{-1}=C\left(\frac{1}{4m},\frac{p-m}{m(3p-4m)}\right).
\end{equation}
Moreover, $C(m,p)\vec M:\vec N=\vec M:C(m,p)\vec N$ and
$$
\min\{ 2\underline{\mathrm{m}}, 3\underline{\mathrm{p}}-4 \overline{\mathrm{m}}\} \, \vec M:\vec M \le C(m,p)\vec M:\vec M\le \max\{ 2\overline{\mathrm{m}}, 3\overline{\mathrm{p}}-4 \underline{\mathrm{m}}\} \,  \vec M:\vec M,
$$
see, e.g., \cite[Lemma~50]{zeltmann_phd}. 
Provided $\rho(x) >0$, $\big(\mu_{0}(x),\pi_{0}(x)  \big), \big(\tau_\sh(x)\mu_0(x), \tau_\pr(x)\pi_0(x)\big)
\in\mathsf{D}(C)$ for almost all $x \in D$ we conclude that
\begin{equation} \label{B}
B\!\begin{pmatrix}
\vec w\\[1mm]
\vec\psi_0\\[1mm]
\vec\psi_1\\[1mm]
\vdots\\[1mm]
\vec\psi_L
\end{pmatrix}\!=\!\begin{pmatrix}
{\rho}\,\vec w\\[1mm]
\widetilde{C}\big(\mu_{0},\pi_{0}  \big)\vec\psi_0\\[1mm]
\widetilde{C}\big(\tau_\sh\mu_0, \tau_\pr\pi_0\big)\vec\psi_1\\[1mm]
\vdots\\[1mm]
\widetilde{C}\big(\tau_\sh\mu_0, \tau_\pr\pi_0\big)\vec\psi_L
\end{pmatrix}
\end{equation}
{\red yielding a uniformly positive} $B\in\mathcal{L}^*(X)$  in the  sense of our general hypotheses from the beginning of Section~\ref{sec:abstract}. Hence, the general hypotheses are satisfied for the viscoelastic wave equation.
\subsection{\red Full waveform forward operator}
In FWI the five parameters $ (\rho,v_{\sh}, \tau_{\sh}, v_{\pr},\tau_{\pr})$ are of interest. Therefore we will define a parameter-to-solution map $\Phi$  which takes these 
parameters as arguments. A physically meaningful domain of definition for $\Phi$ is
\begin{multline*}
\mathsf{D}(\Phi)=\big\{ (\rho,v_{\sh}, \tau_{\sh}, v_{\pr},\tau_{\pr})\in L^\infty(D)^5: \rho_{\min} \le \rho(\cdot)\le \rho_{\max},\ 
 v_{\pr,\min} \le v_\pr(\cdot)\le  v_{\pr,\max}, \\ v_{\sh,\min} \le v_\sh(\cdot)\le  v_{\sh,\max},\ 
 \tau_{\pr,\min} \le \tau_\pr(\cdot)\le  \tau_{\pr,\max}, \ \tau_{\sh,\min} \le \tau_\sh(\cdot)\le  \tau_{\sh,\max} \ \text{a.e.\ in } D\big\}
\end{multline*}
with suitable positive bounds $0<\rho_{\min} < \rho_{\max}<\infty$, etc. 

In view of \eqref{V_sp} we set
$$
\mu_{\min}:=\frac{\rho_{\min}\, v_{\sh,\min}^2}{1+\tau_{\sh,\max}\alpha}\quad \text{and}\quad
\mu_{\max}:=\frac{\rho_{\max}\, v_{\sh,\max}^2}{1+\tau_{\sh,\min}\alpha}
$$
which are induced  lower and upper bounds for $\mu_0$.
We set the bounds   $\pi_{\min}$ and $\pi_{\max}$ for $\pi_0$  accordingly  by replacing $\sh$ by $\pr$. 
Next we define 
$\underline{\mathrm{p}}$, $\overline{\mathrm{p}}$, $\underline{\mathrm{m}}$,  and $\overline{\mathrm{m}}$ such that 
$(\mu_{0},\pi_{0}), (\tau_\sh\mu_0, \tau_\pr\pi_0)$ as functions of 
$(\rho,v_{\pr},
v_{\sh}, \tau_{\pr}, \tau_{\sh})\in \mathsf{D}(\Phi)$ are in $\mathsf{D}(C)$. Indeed, 
$$
\underline{\mathrm{p}}:= \pi_{\min}\,\min\{1, \tau_{\pr,\min}\}\quad\text{and}\quad 
\overline{\mathrm{p}}:= \pi_{\max}\,\max\{1, \tau_{\pr,\max}\}
$$
with $\underline{\mathrm{m}}$ and  $\overline{\mathrm{m}}$ set correspondingly will do the job. The restriction  $3\underline{\mathrm{p}}>4 \overline{\mathrm{m}}$ translates into
$$
\frac{4}{3}\;\frac{\rho_{\max}}{\rho_{\min}}\;\frac{1+\tau_{\pr,\max}\alpha}{1+\tau_{\sh,\min}\alpha}
\;\frac{\max\{1, \tau_{\sh,\max}\}}{\min\{1, \tau_{\pr,\min}\}}< \frac{v_{\pr,\min}^2}{v_{\sh,\max}^2}
$$
which reflects in a way the physical fact that pressure waves propagate considerably faster than shear waves.

For $\vec f\in W^{1,1}([0,T],L^2(D,\RR^3))$ and $u_0=(\vec v(0),\Sig_0(0),\dots,\Sig_L(0))\in \mathsf{D}(A)$ the 
{\em {\red full waveform forward operator}}
$$
\Phi\colon \mathsf{D}(\Phi)\subset L^\infty(D)^5\to L^2([0,T],X),\quad (\rho,v_{\sh}, \tau_{\sh}, v_{\pr},\tau_{\pr})
\mapsto (\vec v,\Sig_0,\dots,\Sig_L),
$$
is well defined where $(\vec v,\Sig_0,\dots,\Sig_L)$ is the unique classical solution of \eqref{C2:elawave_trans} with initial value $u_0$.

To benefit from our abstract results  we factorize $\Phi=F\circ V$ where $F$ is as in \eqref{F} and 
$$
V\colon \mathsf{D}(\Phi)\subset L^\infty(D)^5\to\mathcal{L}^*(X), \quad
 (\rho,v_{\sh}, \tau_{\sh}, v_{\pr},\tau_{\pr})\mapsto B,
$$
where $B$ is defined in \eqref{B} via \eqref{V_p2}. 
\begin{remark} \label{rem_V}
Note that the image of $V$ is in $\mathsf{D}(F)$ by an appropriate choice of
$\beta_-$ and $\beta_+$ in terms of $\rho_{\min}$, $\rho_{\max}$, $\underline{\mathrm{p}}$, $\overline{\mathrm{p}}$, 
$\underline{\mathrm{m}}$,  and $\overline{\mathrm{m}}$. 
\end{remark}
The inverse problem of FWI in the viscoelastic regime is locally ill-posed. This can be proved using Theorem~\ref{t-illposed_4}, compare the proof of Theorem~6.7 of \cite{Kirsch-Rieder16}. We give a direct proof though.
\begin{theorem}
The inverse problem $\Phi(\cdot)=  (\vec v,\Sig_0,\dots,\Sig_L)$ is locally ill-posed at any 
interior point of $\vec p=(\rho,v_{\sh}, \tau_{\sh}, v_{\pr},\tau_{\pr})\in\mathsf{D}(\Phi)$.
\end{theorem}
\begin{proof}
Fix a point $\xi\in D$ and define balls $K_n=\{y\in\RR^3: |y-\xi|\le \delta/n\}$ with a $\delta>0$ so small that $K_n\subset D$
for all $n\in\NN$. Let $\chi_n$ be the indicator function of $K_n$. Further, for any $r>0$ such that 
$\vec p_n:=\vec p+r(\chi_n,\chi_n,\chi_n,\chi_n,\chi_n) \in\mathsf{D}(\Phi)$ we have that $\Vert \vec p_n-\vec p\Vert_{L^\infty(D)^5}=r$, that is,
$\vec p_n$ does not converge to $\vec p$. However, $\lim_{n\to\infty}\Vert \Phi(\vec p_n)-\Phi(\vec p)
\Vert_{L^2([0,T],X)}=0$ as we demonstrate now.

Let $u_n=\Phi(\vec p_n)$ and $u=\Phi(\vec p)$. Then, $d_n=u_n-u$ satisfies 
$$
V(\vec p_n) d_n'+Ad_n + V(\vec p_n)Qd_n=\big(V(\vec p)-V(\vec p_n)\big)(u'+Qu),\quad d_n(0)=0.
$$
By the continuous dependence of $d_n$  on the data, see  \eqref{cont_dep_on_data},  we obtain
$$
\Vert d_n\Vert_{L^2([0,T],X)}\lesssim \big\Vert \big(V(\vec p)-V(\vec p_n)\big)(u'+Qu) \big\Vert_{L^1([0,T],X)}
$$
where the constant is independent of $n$, see Remark~\ref{rem_V}. Next one shows 
$\lim_{n\to\infty}\Vert \big(V(\vec p)-V(\vec p_n)\big)v\Vert_X=0$ for any $v\in X$ using $\vec p_n\to \vec p$ pointwise 
a.e.~ in $D$ as $n\to\infty$ and the dominated convergence theorem. 
Since $\Vert V(\vec p_n)\Vert_X\lesssim 1$ for all 
$n\in\NN$ a further application of the dominated convergence theorem with respect to the time domain yields
$$
\int_0^T \big\Vert \big(V(\vec p)-V(\vec p_n)\big)\big(u'(t)+Qu(t)\big) \big\Vert_X\d t \xrightarrow{n\to\infty} 0
$$
and finishes the proof.
\end{proof}
\subsection{First order differentiability}
To derive the first order Fr\'echet derivative of $\Phi$ we provide the Fr\'echet derivative of 
$V$. Its formulation needs the derivative of $\widetilde{C}$ which we take from 
\cite[Lemma~6.3]{Kirsch-Rieder16}:
\begin{equation}\label{C'}
\widetilde{C}'(m,p)\begin{bmatrix}
\widehat{m}\\
\widehat{p}
\end{bmatrix} =-\widetilde{C}(m,p)\circ C(\widehat{m},\widehat{p})\circ\widetilde{C}(m,p)
\end{equation}
for $(m,p)\in \operatorname{int}(\mathsf{D}(C))$ and $(\widehat{m},\widehat{p})\in\RR^2$.

Let $\vec p=(\rho,v_{\sh}, \tau_{\sh}, v_{\pr},\tau_{\pr})\in   \operatorname{int}(\mathsf{D}(\Phi))$ and  
$\widehat{\vec p}=(\widehat{\rho}, \widehat{v}_{\sh}, \widehat{\tau}_{\sh}, \widehat{v}_{\pr}, \widehat{\tau}_{\pr})\in 
L^\infty(D)^5$. Then, $V'(\vec p)\widehat{\vec p}\in \mathcal{L}^*(X)$ is given by 
\begin{equation}\label{V'}
V'(\vec p)\widehat{\vec p}\!
\begin{pmatrix}\!
\vec w\\[1mm]
\vec\psi_0\\[1mm]
\vdots\\[1mm]
\vec\psi_L
\!\end{pmatrix} = 
\begin{pmatrix}
\widehat{\rho} \,\vec w\\[1mm]
-\frac{\widehat{\rho}}{\rho^2}\widetilde{C}(\mu,\pi)\vec\psi_0 +\frac{1}{\rho} \widetilde{C}'(\mu,\pi)\!
\begin{bmatrix}
\widetilde{\mu}\\[1mm]
\widetilde{\pi}
\end{bmatrix}\!
\vec \psi_0\\[6mm]
-\frac{\widehat{\rho}}{\rho^2}\widetilde{C}(\tau_\sh\mu,\tau_\pr\pi)\vec\psi_1 +\frac{1}{\rho} \widetilde{C}'(\tau_\sh\mu,\tau_\pr\pi)\!
\begin{bmatrix}
\widehat{\mu}\\[1mm]
\widehat{\pi}
\end{bmatrix}\!
\vec \psi_1\\[1mm]
\vdots\\[1mm]
-\frac{\widehat{\rho}}{\rho^2}\widetilde{C}(\tau_\sh\mu,\tau_\pr\pi)\vec\psi_L +\frac{1}{\rho} \widetilde{C}'(\tau_\sh\mu,\tau_\pr\pi)\!
\begin{bmatrix}
\widehat{\mu}\\[1mm]
\widehat{\pi}
\end{bmatrix}\!
\vec \psi_L
\end{pmatrix}
\end{equation}
where $\mu=\mu_0/\rho$,  $\pi=\pi_0/\rho$, see \eqref{V_p2}, and 
\begin{alignat}{2}
\widetilde{\mu} &= \frac{2v_{\sh}}{1+\tau_{\sh}\alpha}\, \widehat{v}_{\sh} - \frac{\alpha\,v_{\sh}^2}{(1+\tau_{\sh}\alpha)^2}\, \widehat{\tau}_{\sh},&\qquad
\widetilde{\pi} &=\frac{2v_{\pr}}{1+\tau_{\pr}\alpha}\, \widehat{v}_{\pr} - \frac{\alpha\,v_{\pr}^2}{(1+\tau_{\pr}\alpha)^2}\, \widehat{\tau}_{\pr},\label{tilde} \\[1mm]
\widehat{\mu} &= \frac{2\tau_\sh\,v_{\sh}}{1+\tau_{\sh}\alpha}\, \widehat{v}_{\sh} + \frac{v_{\sh}^2}{(1+\tau_{\sh}\alpha)^2}\, \widehat{\tau}_{\sh},&\qquad
\widehat{\pi} &=\frac{2\tau_\pr\,v_{\pr}}{1+\tau_{\pr}\alpha}\, \widehat{v}_{\pr} +\frac{v_{\pr}^2}{(1+\tau_{\pr}\alpha)^2}\, \widehat{\tau}_{\pr}.\label{hat}
\end{alignat}
\begin{theorem}\label{th:FWI_first}
Under the assumptions made in this section the {\red full waveform forward operator} $\Phi$ is Fr\'echet differentiable at any interior point 
$\vec p=(\rho,v_{\sh}, \tau_{\sh}, v_{\pr},\tau_{\pr})$ of $\mathsf{D}(\Phi)$: For 
$\widehat{\vec p}=(\widehat{\rho}, \widehat{v}_{\sh}, \widehat{\tau}_{\sh}, \widehat{v}_{\pr}, \widehat{\tau}_{\pr})\in 
L^\infty(D)^5$ we have  $\Phi'(\vec p)\widehat{\vec p}= \overline{u}$ where $\overline{u}=(\overline{\vec v},\overline{\vec \sigma}_0,\dotsc,\overline{\vec \sigma}_L)\in \mathcal{C}([0,T],X)$ with $\overline{u}(0)=0$ is the mild solution of
\begin{subequations}\label{FWI_first}
\begin{align}
\rho\,\partial_t \overline{\vec v} &= \Div\Big(\sum_{l=0}^L \overline{\vec \sigma}_l\Big)- \widehat{\rho}\, \partial_t\vec v,\\[1mm]
\partial_t \overline{\vec \sigma}_0 &= C(\mu_0,\pi_0)\vec \varepsilon(\overline{\vec v} ) + \big(\widehat{\rho}\,
C(\mu,\pi)+\rho\,C(\widetilde{\mu},\widetilde{\pi})\big) \varepsilon(\vec v), \label{FWI_first2}\\[2mm]
\partial_t \overline{\vec \sigma}_l &=  C(\tau_\sh\mu_0,\tau_\pr\pi_0)\vec \varepsilon(\overline{\vec v})  \label{FWI_first3}\\
&\qquad\quad-
\frac{1}{\tau_{\Sig,l}}\,\overline{\vec \sigma}_l + \big(\widehat{\rho}\,
C(\tau_\sh\mu,\tau_\pr\pi)+\rho\,C(\widehat{\mu},\widehat{\pi})\big) \varepsilon(\vec v),\quad l=1,\dotsc,L,\nonumber
\end{align}
\end{subequations}
where $(\vec v,\vec \sigma_0,\dotsc,\vec \sigma_L)$ is the classical solution of \eqref{C2:elawave_trans}.
\end{theorem}
\begin{proof}
We apply Theorem~\ref{th:first} to $\Phi'(\vec p)\widehat{\vec p} = F'(V(\vec p))V'(\vec p)\widehat{\vec p}$ and get the system
\begin{multline*}
\begin{pmatrix}
{\rho}\,\partial_t \overline{\vec v}\\[1mm]
\frac{1}{\rho}\,\widetilde{C}(\mu,\pi )\partial_t \overline{\vec \sigma}_0\\[1mm]
\frac{1}{\rho}\,\widetilde{C}(\tau_\sh\mu, \tau_\pr\pi)\partial_t \overline{\vec \sigma}_1\\[1mm]
\vdots\\[1mm]
\frac{1}{\rho}\,\widetilde{C}(\tau_\sh\mu, \tau_\pr\pi)\partial_t \overline{\vec \sigma}_L
\end{pmatrix}
=
\begin{pmatrix}
\Div \big(\sum_{l=0}^{L}\overline{\vec\sigma}_l\big) \\[1mm] \varepsilon(\overline{\vec v})\\[1mm] \vdots\\[1mm] \varepsilon(\overline{\vec v})\end{pmatrix}-
\begin{pmatrix}
\vec 0\\[1mm]
\vec 0\\[1mm]
\frac{1}{\rho\,\tau_{\Sig,1}}\,\widetilde{C}\big(\tau_\sh\mu, \tau_\pr\pi\big)\overline{\vec\sigma}_1\\[1mm]
\vdots\\[1mm]
\frac{1}{\rho\,\tau_{\Sig,L}}\,\widetilde{C}\big(\tau_\sh\mu, \tau_\pr\pi\big)\overline{\vec\sigma}_L
\end{pmatrix}\\
-V'(\vec p)\widehat{\vec p}\left[  \begin{pmatrix}
\partial_t \vec v\\[1mm]
\partial_t \vec \sigma_0\\[1mm]
\partial_t \vec \sigma_1\\[1mm]
\vdots\\[1mm]
\partial_t \vec \sigma_L
\end{pmatrix} 
+
\begin{pmatrix}
\vec 0\\[1mm]
\vec 0\\[1mm]
\frac{1}{\tau_{\Sig,1}}\,\vec\sigma_1\\[1mm]
\vdots\\[1mm]
\frac{1}{\tau_{\Sig,L}}\,\vec\sigma_L
\end{pmatrix}
\right]
\end{multline*}
which is equivalent to \eqref{FWI_first} in view of \eqref{C2:elawave2_trans},  \eqref{C2:elawave3_trans}, \eqref{C'}, and \eqref{V'}.
\end{proof}
\begin{theorem}\label{th:FWI_adjoint}
The assumptions are as in Theorem~{\rm\ref{th:FWI_first}}. Then, 
the adjoint $\Phi'(\vec p)^*\in \mathcal{L}\big( L^2([0,T],X),( L^\infty(D)^5)'\big)$ at $\vec p=(\rho,v_{\sh}, \tau_{\sh}, v_{\pr},\tau_{\pr})\in \mathsf{D}(\Phi)$ is given by 
$$
\Phi'(\vec p)^*\vec g = 
\begin{pmatrix}
\int_0^T \big(\partial_t\vec v \cdot {\vec w}  - \frac{1}{\rho}\,\vec\varepsilon(\vec v):( {\vec \varphi}_0 +\vec\Sigma)\big)\d t \\[2mm]
 \frac{2}{v_\sh} \int_0^T\big(-\vec\varepsilon(\vec v): ( {\vec \varphi}_0 +\vec\Sigma)+\pi\trace (\vec\Sigma^v)\Div\vec v\big)  \d t\\[2mm]
\frac{1}{1+\alpha\tau_\sh}\int_0^T \big( \vec\varepsilon(\vec v):\vec\Sigma^\tau_{\sh,2} +\pi\trace( \vec\Sigma^\tau_{\sh,1})\Div\vec v\big) \d t
\\[2mm]
-\frac{2\pi}{v_\pr}\int_0^T \trace(\vec\Sigma^v) \Div\vec v\, \d t
\\[2mm]
\frac{\pi}{1+\alpha\tau_\pr}\int_0^T \trace(\vec\Sigma^\tau_\pr) \Div\vec v\, \d t
\end{pmatrix} \in L^1(D)^5
$$
for $\vec g =(\vec g_{-1},\vec g_0,\dotsc,\vec g_L)\in L^2\big([0,T],L^2(D,\RR^3)\times L^2(D,\RS)^{1+L}\big)$
where  $\vec v$ is the first component of the solution of \eqref{C2:elawave_trans}, $\vec\Sigma =\sum_{l=1}^L {\vec \varphi}_l$, and
\begin{align*}
\vec\Sigma^v&=\frac{1}{3\pi-4\mu}\, {\vec \varphi}_0 + 
\frac{\tau_\pr}{3\tau_\pr\pi-4\tau_\sh\mu} \,\vec\Sigma,\\[1mm]
\vec\Sigma^\tau_{\sh,1}&=-\frac{\alpha}{3\pi-4\mu}\, {\vec \varphi}_0 
+\frac{\tau_\pr}{\tau_\sh(3\tau_\pr\pi-4\tau_\sh\mu)}\,\vec\Sigma,\quad
\vec\Sigma^\tau_{\sh,2}=\alpha\, {\vec \varphi}_0 
-\frac{1}{\tau_\sh}\,\vec\Sigma,\\[1mm]
\vec\Sigma^\tau_\pr&=\frac{\alpha}{3\pi-4\mu}\, {\vec \varphi}_0 
-\frac{1}{3\tau_\pr\pi-4\tau_\sh\mu}\,\vec\Sigma,
\end{align*}
and $w=({\vec w},  {\vec \varphi}_0,\dotsc, {\vec \varphi}_L)\in \mathcal{C}([0,T],X)$ uniquely solves
\begin{subequations}
      \label{C2:elawave_trans_ad}
    \begin{align}
      \partial_t \vec w&  = \frac{1}{\rho}\Div \Big(\sum_{l=0}^{L}{\vec\varphi}_l\Big)+\frac{1}{\rho}\,\vec g_{-1},
        \label{C2:elawave1_trans_ad}\\
        \partial_t {\vec\varphi}_0 &
        = C\big(\mu_{0},\pi_{0}  \big)\big(\varepsilon(\vec w)+\vec g_0\big),\label{C2:elawave2_trans_ad}\\
          \partial_t \vec\varphi_l &
          =C\big(\tau_\sh\mu_0, \tau_\pr\pi_0\big)\big(\varepsilon(\vec w)+\vec g_l\big) +\frac{1}{\tsigl}\,{\vec\varphi}_l,\qquad l=1,\dotsc,L,
            \label{C2:elawave3_trans_ad}
    \end{align}
    \end{subequations}
    with $w(T)=0$.
\end{theorem}
\begin{remark}
Please note that $\Phi'(\vec p)^*$ actually  maps  into $ L^1(D)^5$ which is a subspace of $(L^\infty(D)^5)'$. This remark applies also  to the adjoints considered in Theorems~{\rm\ref{FWI''_adjoint1}} and~{\rm\ref{FWI''_adjoint2}} below.
\end{remark}
\begin{proof}[Proof of Theorem~{\rm\ref{th:FWI_adjoint}}]
Using $A^*=-A$ (skew-symmetry), $Q^*=Q$, and $QB=BQ$ we convince ourselves that \eqref{C2:elawave_trans_ad}
is the concrete version of the abstract equation \eqref{evolution_back}.
Further, by Theorem~\ref{th:adjoint}, 
\begin{align}
\big\langle \Phi'(\vec p)^*\vec g,\widehat{\vec p}\big\rangle_{( L^\infty(D)^5)'\times  L^\infty(D)^5}
&= \big\langle F'(V(\vec p))^*\vec g, V'(\vec p)\widehat{\vec p}\big\rangle_{\mathcal{L}(X)'\times \mathcal{L}(X)}
\nonumber\\[1mm]
&=\int_0^T\! \big\langle V'(\vec p)\widehat{\vec p}\big(u'(t) +Qu(t)\big),w(t)\rangle_X\,\d t\label{eq:aux_FWI_ad}
\end{align}
where $u=(\vec v,\vec \sigma_0,\dotsc,\vec \sigma_L)$ is the classical solution of \eqref{C2:elawave_trans}. We are now going to evaluate the  above integrand suppressing its $t$-dependence. Using \eqref{V'} and \eqref{C'} we find 
for $\widehat{\vec p}=(\widehat{\rho}, \widehat{v}_{\sh}, \widehat{\tau}_{\sh}, \widehat{v}_{\pr}, \widehat{\tau}_{\pr})$ that
\begin{equation}\label{sum}
\big\langle V'(\vec p)\widehat{\vec p}\big(u' +Qu\big),w\rangle_X =\int_D\big(\widehat{\rho}\, \partial_t\vec v \cdot {\vec w} 
+S_0+S_1+\cdots + S_L\big)\d x
\end{equation}
with 
$$
S_0= \Big[-\frac{\widehat{\rho}}{\rho^2}\widetilde{C}(\mu,\pi)\partial_t\Sig_0 -\frac{1}{\rho} \widetilde{C}(\mu,\pi)
C(\widetilde{\mu},\widetilde{\pi}) \widetilde{C}(\mu,\pi) \partial_t\Sig_0\Big]: {\vec \varphi}_0
$$
and, for $l=1,\dotsc,L$, 
\begin{multline*}
S_l=\Big[-\frac{\widehat{\rho}}{\rho^2}\widetilde{C}(\tau_\sh\mu,\tau_\pr\pi)\Big(\partial_t\Sig_l +\frac{\Sig_l}{\tau_{\Sig,l}}\Big)\\
-\frac{1}{\rho} \widetilde{C}(\tau_\sh\mu,\tau_\pr\pi)
C(\widehat{\mu},\widehat{\pi}) \widetilde{C}(\tau_\sh\mu,\tau_\pr\pi) \Big(\partial_t\Sig_l +\frac{\Sig_l}{\tau_{\Sig,l}}\Big)\Big]: {\vec \varphi}_l.
\end{multline*}
In view of \eqref{C2:elawave2_trans}   we may write
$$
S_0= \Big[-\frac{\widehat{\rho}}{\rho}\vec\varepsilon(\vec v) - \widetilde{C}(\mu,\pi)
C(\widetilde{\mu},\widetilde{\pi}) \vec\varepsilon(\vec v) \Big]: {\vec \varphi}_0 = 
-\frac{\widehat{\rho}}{\rho}\vec\varepsilon(\vec v):{\vec \varphi}_0 - 
C(\widetilde{\mu},\widetilde{\pi}) \vec\varepsilon(\vec v):\widetilde{C}(\mu,\pi) {\vec \varphi}_0
$$
and, similarly by \eqref{C2:elawave3_trans} ,
$$
S_l = 
-\frac{\widehat{\rho}}{\rho}\vec\varepsilon(\vec v):{\vec \varphi}_l - 
C(\widehat{\mu},\widehat{\pi}) \vec\varepsilon(\vec v):\widetilde{C}(\tau_\sh\mu,\tau_\pr\pi) {\vec \varphi}_l,
\quad l=1,\dotsc,L.
$$
Next, using  \eqref{C^-1}, we compute
\begin{align}
&C(\widetilde{\mu},\widetilde{\pi}) \vec\varepsilon(\vec v):\widetilde{C}(\mu,\pi) {\vec \varphi}_0\nonumber\\
&\qquad\quad= \big( 2\widetilde{\mu}\,  \vec\varepsilon(\vec v) + (\widetilde{\pi}-2\widetilde{\mu} )\Div\V\, \vec I\big):
\Big( \frac{1}{2\mu}\,{\vec \varphi}_0 + \frac{2\mu-\pi}{2\mu(3\pi-4\mu)}\,\trace({\vec \varphi}_0)\vec I\Big)\label{eq:aux_FWI_ad2}\\
&\qquad\quad= \widetilde{\mu}\Big(\frac{1}{\mu } \vec\varepsilon(\vec v):{\vec \varphi}_0
- \frac{\pi}{\mu(3\pi-4\mu)}\,\Div\V\, \trace({\vec \varphi}_0)\Big)+\frac{\widetilde{\pi}}{3\pi-4\mu} 
\,\Div\V\, \trace({\vec \varphi}_0)\nonumber
\end{align}
yielding
\begin{multline*}
S_0= -\frac{\widehat{\rho}}{\rho}\vec\varepsilon(\vec v):{\vec \varphi}_0\\ +
\widetilde{\mu}\Big(-\frac{1}{\mu } \vec\varepsilon(\vec v):{\vec \varphi}_0
+\frac{\pi}{\mu(3\pi-4\mu)}\,\Div\V\, \trace({\vec \varphi}_0)\Big)-\frac{\widetilde{\pi}}{3\pi-4\mu} 
\,\Div\V\, \trace({\vec \varphi}_0).
\end{multline*}
Analogously,
\begin{multline*}
S_l= -\frac{\widehat{\rho}}{\rho}\vec\varepsilon(\vec v):{\vec \varphi}_l\\ +
\widehat{\mu}\Big(-\frac{1}{\tau_\sh\mu } \vec\varepsilon(\vec v):{\vec \varphi}_l
+\frac{\tau_\pr\pi}{\tau_\sh\mu(3\tau_\pr\pi-4\tau_\sh\mu)}\,\Div\V\, \trace({\vec \varphi}_l)\Big)-\frac{\widehat{\pi}}{3\tau_\pr\pi-4\tau_\sh\mu} 
\,\Div\V\, \trace({\vec \varphi}_l). 
\end{multline*}
Next we group the terms in the sum \eqref{sum} belonging to the five components of $\widehat{\vec p}$. To this end 
we replace $\widetilde{\mu}$, $\widetilde{\pi}$, $\widehat{\mu}$, and $\widehat{\pi}$ by their respective expressions from 
\eqref{tilde} and \eqref{hat} which we slightly rewrite introducing $\mu$ and $\pi$:
\begin{alignat}{2}
\widetilde{\mu} &= \frac{2\mu}{v_\sh}\, \widehat{v}_{\sh} - \frac{\alpha\,\mu}{1+\tau_{\sh}\alpha}\, \widehat{\tau}_{\sh},&\qquad
\widetilde{\pi} &=\frac{2\pi}{v_{\pr}}\, \widehat{v}_{\pr} - \frac{\alpha\,\pi}{1+\tau_{\pr}\alpha}\, \widehat{\tau}_{\pr},
\label{eq:aux_FWI_ad3}\\[2mm]
\widehat{\mu} &= \frac{2\tau_\sh\,\mu}{v_{\sh}}\, \widehat{v}_{\sh} + \frac{\mu}{1+\tau_{\sh}\alpha}\, \widehat{\tau}_{\sh},&\qquad
\widehat{\pi} &=\frac{2\tau_\pr\,\pi}{v_{\pr}}\, \widehat{v}_{\pr} +\frac{\pi}{1+\tau_{\pr}\alpha}\, \widehat{\tau}_{\pr}.
\label{eq:aux_FWI_ad4}
\end{alignat}
After some algebra we get
\begin{multline*}
\big\langle V'(\vec p)\widehat{\vec p}\big(u' +Qu\big),\overline{u}\rangle_X = \int_D\Big[\widehat{\rho}\Big(\partial_t\vec v \cdot {\vec w}  - \frac{1}{\rho}\,\vec\varepsilon(\vec v):({\vec \varphi}_0 +\vec\Sigma)\Big)\\ +
\widehat{v}_\sh\,\frac{2}{v_\sh}\Big( -\,\vec\varepsilon(\vec v): ( {\vec \varphi}_0 +\vec\Sigma)+\pi\trace (\vec\Sigma^v)\Div\vec v \Big)\\+ \frac{\widehat{\tau}_\sh}{1+\alpha\tau_\sh}\Big( \vec\varepsilon(\vec v):\vec\Sigma^\tau_{\sh,2} +\pi\trace( \vec\Sigma^\tau_{\sh,1})\Div\vec v  \Big) \\
-\widehat{v}_\pr\,\frac{2\pi}{v_\pr}\trace(\vec\Sigma^v) \Div\vec v +\widehat{\tau}_\pr\,\frac{\pi}{1+\alpha\tau_\pr}\trace(\vec\Sigma^\tau_\pr) \Div\vec v\Big]\d x \qquad
\end{multline*}
which ends the proof.
\end{proof}
\subsection{Second order differentiability}
The second derivative of $\Phi$ is given by
\begin{equation}\label{Phi''}
\Phi''(\vec p)[\widehat{\vec p}_1,\widehat{\vec p}_2]=F''(V(\vec p))[V'(\vec p)\widehat{\vec p}_1,V'(\vec p)\widehat{\vec p}_2]
+ F'(V(\vec p))V''(\vec p)[\widehat{\vec p}_1,\widehat{\vec p}_2]
\end{equation}
using the chain and product rules, see, e.g., \cite[Section~4.3]{zeidler1}. In a first step we need to find~$V''$. 
Differentiating \eqref{V'} at $\vec p=(\rho,v_{\sh}, \tau_{\sh}, v_{\pr},\tau_{\pr})\in   \operatorname{int}(\mathsf{D}(\Phi))$ we obtain
\begin{multline}\label{V''}
V''(\vec p)[\widehat{\vec p}_1,\widehat{\vec p}_2]\!
\begin{pmatrix}\!
\vec w\\[1mm]
\vec\psi_0\\[1mm]
\vdots\\[1mm]
\vec\psi_L
\!\end{pmatrix} =\\
\begin{pmatrix}
\vec 0\\[1mm]
\Big(\frac{\widehat{\rho}_1\widehat{\rho}_2}{\rho^3} \widetilde{C}(\mu,\pi)
-\frac{\widehat{\rho}_1}{\rho^2}\widetilde{C}'(\mu,\pi)\begin{bmatrix}
\widetilde{\mu}_2\\[1mm]
\widetilde{\pi}_2
\end{bmatrix} -\frac{\widehat{\rho}_2}{\rho^2} \widetilde{C}'(\mu,\pi)\!
\begin{bmatrix}
\widetilde{\mu}_1\\[1mm]
\widetilde{\pi}_1
\end{bmatrix}\! +\frac{1}{\rho} \widetilde{C}''(\mu,\pi)\!
\begin{bmatrix}
\widetilde{\mu}_1\\[1mm]
\widetilde{\pi}_1
\end{bmatrix}
\begin{bmatrix}
\widetilde{\mu}_2\\[1mm]
\widetilde{\pi}_2
\end{bmatrix}\Big)\vec\psi_0
\\[6mm]
\Big(\frac{\widehat{\rho}_1\widehat{\rho}_2}{\rho^3} \widetilde{C}(\tau_\sh\mu,\tau_\pr\pi)
-\frac{\widehat{\rho}_1}{\rho^2}\widetilde{C}'(\tau_\sh\mu,\tau_\pr\pi)\begin{bmatrix}
\widehat{\mu}_2\\[1mm]
\widehat{\pi}_2
\end{bmatrix} -\frac{\widehat{\rho}_2}{\rho^2} \widetilde{C}'(\tau_\sh\mu,\tau_\pr\pi)\!
\begin{bmatrix}
\widehat{\mu}_1\\[1mm]
\widehat{\pi}_1
\end{bmatrix}\! \\[4mm] \qquad\qquad\qquad \qquad\qquad\qquad\qquad\qquad\qquad\qquad+\frac{1}{\rho} \widetilde{C}''(\tau_\sh\mu,\tau_\pr\pi)\!
\begin{bmatrix}
\widehat{\mu}_1\\[1mm]
\widehat{\pi}_1
\end{bmatrix}
\begin{bmatrix}
\widehat{\mu}_2\\[1mm]
\widehat{\pi}_2
\end{bmatrix}\Big)\vec\psi_1\\
\vdots\\[1mm]
\Big(\frac{\widehat{\rho}_1\widehat{\rho}_2}{\rho^3} \widetilde{C}(\tau_\sh\mu,\tau_\pr\pi)
-\frac{\widehat{\rho}_1}{\rho^2}\widetilde{C}'(\tau_\sh\mu,\tau_\pr\pi)\begin{bmatrix}
\widehat{\mu}_2\\[1mm]
\widehat{\pi}_2
\end{bmatrix} -\frac{\widehat{\rho}_2}{\rho^2} \widetilde{C}'(\tau_\sh\mu,\tau_\pr\pi)\!
\begin{bmatrix}
\widehat{\mu}_1\\[1mm]
\widehat{\pi}_1
\end{bmatrix}\! \\[4mm] \qquad\qquad\qquad \qquad\qquad\qquad\qquad\qquad\qquad\qquad+\frac{1}{\rho} \widetilde{C}''(\tau_\sh\mu,\tau_\pr\pi)\!
\begin{bmatrix}
\widehat{\mu}_1\\[1mm]
\widehat{\pi}_1
\end{bmatrix}
\begin{bmatrix}
\widehat{\mu}_2\\[1mm]
\widehat{\pi}_2
\end{bmatrix}\Big)
\vec \psi_L
\end{pmatrix}
\end{multline}
for   
$\widehat{\vec p}_i=(\widehat{\rho}_i, \widehat{v}_{\sh,i}, \widehat{\tau}_{\sh,i}, \widehat{v}_{\pr,i}, \widehat{\tau}_{\pr,i})\in 
L^\infty(D)^5$, $i=1,2$. Further, $\widetilde{\mu}_i$, $\widetilde{\pi}_i$, and $\widehat{\mu}_i$,  $\widehat{\pi}_i$ are defined as 
in \eqref{tilde} and \eqref{hat}, respectively, plugging in the respective components of $\widehat{\vec p}_i$. We close the expression 
for $V''$ by 
\begin{multline}\label{C''}
\widetilde{C}''(m,p)\!
\begin{bmatrix}
\widehat{m}_1\\[1mm]
\widehat{p}_1
\end{bmatrix}
\begin{bmatrix}
\widehat{m}_2\\[1mm]
\widehat{p}_2
\end{bmatrix} =
\widetilde{C}(m,p)\circ {C}(\widehat{m}_1,\widehat{p}_1)\circ \widetilde{C}(m,p)\circ {C}(\widehat{m}_2,\widehat{p}_2)\circ \widetilde{C}(m,p)\\
+\widetilde{C}(m,p)\circ {C}(\widehat{m}_2,\widehat{p}_2)\circ \widetilde{C}(m,p)\circ {C}(\widehat{m}_1,\widehat{p}_1)\circ \widetilde{C}(m,p).
\end{multline}
The proof of \eqref{C''} requires straightforward but lengthy calculations.
\begin{theorem}\label{th:FWI_second}
Let $\vec f$ be in $W^{3,1}([0,T],L^2(D,\RR^3))$ with $\vec f(0)=\vec f'(0)=\vec f''(0)=0$. Further, let $u_0=0$ and adopt the assumptions and notation made in this section. 

Then, the {\red full waveform forward operator} $\Phi$ is  twice Fr\'echet differentiable at any interior point 
$\vec p=(\rho,v_{\sh}, \tau_{\sh}, v_{\pr},\tau_{\pr})$ of $\mathsf{D}(\Phi)$: For 
$\widehat{\vec p}_i=(\widehat{\rho}_i, \widehat{v}_{\sh,i}, \widehat{\tau}_{\sh.i}, \widehat{v}_{\pr,i}, \widehat{\tau}_{\pr,i})\in 
L^\infty(D)^5$, $i=1,2$, we have  $\Phi''(\vec p)[\widehat{\vec p}_1,\widehat{\vec p}_2]= {v}+
\overline{\overline{u}}$ where ${v}=(\vec w,\vec \psi_0,\dotsc,\vec \psi_L)$ and 
$\overline{\overline{u}}=(\overline{\overline{\vec v}},\overline{\overline{\vec \sigma}}_0,\dotsc,\overline{\overline{\vec 
\sigma}}_L)$ are both in $\mathcal{C}([0,T],X)$. They are uniquely determined as mild solutions of the following viscoelastic 
equations.

The equations for $\overline{\overline{u}}$ are $\overline{\overline{u}}(0)=0$ and
\begin{align*}
\rho\,\partial_t \overline{ \overline{\vec v}} &= \Div\Big(\sum_{l=0}^L  \overline{\overline{\vec \sigma}}_l\Big)- \widehat{\rho}_1\, \partial_t \overline{\vec v}_2 - \widehat{\rho}_2\, \partial_t \overline{\vec v}_1,\\[1mm]
\partial_t  \overline{\overline{\vec \sigma}}_0 &= C(\mu_0,\pi_0)\vec \varepsilon( \overline{\overline{\vec v} }) + \big(\widehat{\rho}_1\,
C(\mu,\pi)+\rho\,C(\widetilde{\mu}_1,\widetilde{\pi}_1)\big) \vec\varepsilon( \overline{\vec v}_2)\\
&\qquad\quad+ \big(\widehat{\rho}_2\,
C(\mu,\pi)+\rho\,C(\widetilde{\mu}_2,\widetilde{\pi}_2)\big) \vec\varepsilon( \overline{\vec v}_1)
,\\[2mm]
\partial_t  \overline{\overline{\vec \sigma}}_l &=  C(\tau_\sh\mu_0,\tau_\pr\pi_0)\vec \varepsilon(\overline{ \overline{\vec v}})-
\frac{1}{\tau_{\Sig,l}}\,\overline{ \overline{\vec \sigma}}_l \\
&\qquad\quad+ \big(\widehat{\rho}_1\,
C(\tau_\sh\mu,\tau_\pr\pi)+\rho\,C(\widehat{\mu}_1,\widehat{\pi}_1)\big) \vec\varepsilon( \overline{\vec v}_2)\\
&\qquad\quad
+ \big(\widehat{\rho}_2\,
C(\tau_\sh\mu,\tau_\pr\pi)+\rho\,C(\widehat{\mu}_2,\widehat{\pi}_2)\big) \vec\varepsilon( \overline{\vec v}_1)
,\quad l=1,\dotsc,L,\nonumber
\end{align*}
with $\overline{\vec v}_i$ being the first component  of the solution of
\eqref{FWI_first} where the parameters $\widehat{\vec p}$ have to be replaced by $\widehat{\vec p}_i$, $i=1,2$.

The equations for $v$ are $v(0)=0$ and 
\begin{align*}
\rho\,\partial_t \vec w &= \Div\Big(\sum_{l=0}^L  \vec \psi_l\Big),\\[1mm]
\partial_t  \vec \psi_0 &= C(\mu_0,\pi_0)\vec \varepsilon(\vec w) -\Big( \frac{\widehat{\rho}_1\widehat{\rho}_2}{\rho^2}
C(\mu,\pi)+\widehat{\rho}_1C(\widetilde{\mu}_1,\widetilde{\pi}_1)+\widehat{\rho}_2C(\widetilde{\mu}_2,\widetilde{\pi}_2)\\
&\qquad\quad +\rho C(\widetilde{\mu}_1,\widetilde{\pi}_1)\widetilde{C}(\mu,\pi)C(\widetilde{\mu}_2,\widetilde{\pi}_2)
+\rho C(\widetilde{\mu}_2,\widetilde{\pi}_2)\widetilde{C}(\mu,\pi)C(\widetilde{\mu}_1,\widetilde{\pi}_1)\Big) 
\vec\varepsilon(\vec v),\\[2mm]
\partial_t  \vec \psi_l &=  C(\tau_\sh\mu_0,\tau_\pr\pi_0)\vec \varepsilon(\vec w)-\frac{1}{\tau_{\Sig,l}}\,\vec \psi_l-\Big( \frac{\widehat{\rho}_1\widehat{\rho}_2}{\rho^2}
C(\tau_\sh\mu,\tau_\pr\pi)+\widehat{\rho}_1C(\widehat{\mu}_1,\widehat{\pi}_1)+\widehat{\rho}_2C(\widehat{\mu}_2,\widehat{\pi}_2)\\
&\qquad\quad +\rho C(\widehat{\mu}_1,\widehat{\pi}_1)\widetilde{C}(\tau_\sh\mu,\tau_\pr\pi)C(\widehat{\mu}_2,\widehat{\pi}_2)
+\rho C(\widehat{\mu}_2,\widehat{\pi}_2)\widetilde{C}(\tau_\sh\mu,\tau_\pr\pi)C(\widehat{\mu}_1,\widehat{\pi}_1)\Big) 
\vec\varepsilon(\vec v),\nonumber
\end{align*}
 $l=1,\dotsc,L$, where $\vec v$ is the first component of the solution of \eqref{C2:elawave_trans}.
\end{theorem}
\begin{proof}
By \eqref{Phi''},  $\Phi''(\vec p)[\widehat{\vec p}_1,\widehat{\vec p}_2]= v+
\overline{\overline{u}}$ where 
$$
v:=  F'(V(\vec p))V''(\vec p)[\widehat{\vec p}_1,\widehat{\vec p}_2]\quad\text{and}\quad
\overline{\overline{u}}:=F''(V(\vec p))[V'(\vec p)\widehat{\vec p}_1,V'(\vec p)\widehat{\vec p}_2].
$$
We apply Theorems~\ref{th:first} and \ref{th:second} to specify the equations for $v$ and
$\overline{\overline{u}}$, respectively. 

We start with $\overline{\overline{u}}$ which is determined by two coupled equations of type \eqref{FWI_first}. These equations only differ in the plugged in parameters and right hand sides.

Theorem~\ref{th:first} yields the following system for $v$: 
\begin{multline*}
\begin{pmatrix}
{\rho}\,\partial_t \vec w\\[1mm]
\frac{1}{\rho}\,\widetilde{C}(\mu,\pi )\partial_t \vec \psi_0\\[1mm]
\frac{1}{\rho}\,\widetilde{C}(\tau_\sh\mu, \tau_\pr\pi)\partial_t \vec \psi_1\\[1mm]
\vdots\\[1mm]
\frac{1}{\rho}\,\widetilde{C}(\tau_\sh\mu, \tau_\pr\pi)\partial_t \vec \psi_L
\end{pmatrix}
=
\begin{pmatrix}
\Div \big(\sum_{l=0}^{L}{\vec\psi}_l\big) \\[1mm] \vec\varepsilon(\vec w)\\[1mm] \vdots\\[1mm]\vec \varepsilon(\vec w)\end{pmatrix}-
\begin{pmatrix}
\vec 0\\[1mm]
\vec 0\\[1mm]
\frac{1}{\rho\,\tau_{\Sig,1}}\,\widetilde{C}\big(\tau_\sh\mu, \tau_\pr\pi\big)\vec\psi_1\\[1mm]
\vdots\\[1mm]
\frac{1}{\rho\,\tau_{\Sig,L}}\,\widetilde{C}\big(\tau_\sh\mu, \tau_\pr\pi\big)\vec\psi_L
\end{pmatrix}\\
-V''(\vec p)[\widehat{\vec p}_1,\widehat{\vec p}_2]\left[  \begin{pmatrix}
\partial_t \vec v\\[1mm]
\partial_t \vec \sigma_0\\[1mm]
\partial_t \vec \sigma_1\\[1mm]
\vdots\\[1mm]
\partial_t \vec \sigma_L
\end{pmatrix} 
+
\begin{pmatrix}
\vec 0\\[1mm]
\vec 0\\[1mm]
\frac{1}{\tau_{\Sig,1}}\,\vec\sigma_1\\[1mm]
\vdots\\[1mm]
\frac{1}{\tau_{\Sig,L}}\,\vec\sigma_L
\end{pmatrix}
\right]\!.
\end{multline*}
Applying \eqref{C2:elawave2_trans},  \eqref{C2:elawave3_trans}, \eqref{C'},  \eqref{V''}, and \eqref{C''} leads to the equations for
$v$.
\end{proof}
\subsection{An additional adjoint}
As explained in the introduction, second-degree Newton solvers might resolve the cross-talk effect.
In our group we plan to implement a variant of the second-degree Newton method for  viscolesatic FWI. There one needs to solve  a linear system containing 
the operator $\Phi''(\vec p)[\widehat{\vec p},\cdot]$. Our regularization method of choice is the conjugate gradient iteration
which needs the adjoint operator. In this subsection we derive an explicit expression for it.

Recall from \eqref{Phi''} that 
\begin{equation}\label{Phi''_2}
\Phi''(\vec p)[\widehat{\vec p},\cdot]=F''(V(\vec p))[V'(\vec p)\widehat{\vec p},V'(\vec p)\,\cdot]
+ F'(V(\vec p))V''(\vec p)[\widehat{\vec p},\cdot].
\end{equation}
In a first step we therefore consider $F''(B)[H,\cdot]\colon \mathcal{L}^*(X)\to L^2([0,T],X)$ for  $B\in\mathsf{D}(F)$ and 
$H\in \mathcal{L}^*(X)$.
\begin{theorem}\label{th:adjoint2}
Under the assumptions of Theorem~{\rm\ref{th:second}} we have
$$ \big[F''({B})[H_1,\cdot]^\ast g\big]H_2= \int_0^T\!\!\big\langle H_1({u}_2^\prime(t)+Q{u}_2(t))
+ H_2({u}_1^\prime(t)+Q{u}_1(t)),w(t)\big\rangle_X\d t
$$
for $g\in L^2([0,T],X)$,  $H_i\in\mathcal{L}^*(X)$, where  ${u}_i=F'(B)H_i$, $i=1,2$, is the solution  of~\eqref{evolution4}. Further,
$w\in \mathcal{C}([0,T],X)$ is the mild solution of the {\red adjoint} evolution equation
\begin{equation*} 
{B}w^\prime(t) - A^\ast w(t) -Q^*Bw(t)=g(t),\ t\in\,[0,T],\quad w(T)=0.
\end{equation*}
\end{theorem} 
\begin{proof}
Since $\big[F''({B})[H_1,\cdot]^\ast g\big]H_2=\langle \overline{\overline{u}},g\rangle_{ L^2([0,T],X)}$ 
where $ \overline{\overline{u}}\in \mathcal{C}^1([0,T],X)$ solves \eqref{evolution3} we argue similar to the  proof of Theorem~3.8 in \cite{Kirsch-Rieder16}: 
{\red
Assume $g\in W^{1,1}([0,T],X)$. Then, $w\in \mathcal{C}^1([0,T],X)$. Using the selfadjointness of $B$ and integration by parts we compute
\begin{align*}
\langle \overline{\overline{u}},g\rangle_{ L^2([0,T],X)}& =\int_0^T\!\! \langle  \overline{\overline{u}}(t),g(t)\rangle_X\d t
= \int_0^T\!\! \langle  \overline{\overline{u}}(t),{B}w^\prime(t) - A^\ast w(t) -Q^*Bw(t)\rangle_X\d t\\[1mm]
&= -\int_0^T \!\! \langle   B\overline{\overline{u}}\,'(t)+A\overline{\overline{u}}(t)+BQ\overline{\overline{u}}(t),w(t)\rangle_X\d t
\\[1mm]
&=\int_0^T\!\!\big\langle H_1({u}_2^\prime(t)+Q{u}_2(t))
+ H_2({u}_1^\prime(t)+Q{u}_1(t)),w(t)\big\rangle_X\d t.
\end{align*}
The assertion follows since $W^{1,1}([0,T],X)$ is dense in  $L^2([0,T],X)$.
}\end{proof}
\begin{theorem}\label{FWI''_adjoint1}
Under the assumptions of Theorem~{\rm\ref{th:FWI_second}} we have that
the adjoint $$F''(V(\vec p))[V'(\vec p)\widehat{\vec p},V'(\vec p)\,\cdot]^*\in \mathcal{L}\big( L^2([0,T],X),( L^\infty(D)^5)'\big)$$ at $\vec p=(\rho,v_{\sh}, \tau_{\sh}, v_{\pr},\tau_{\pr})\in \mathsf{D}(\Phi)$ and $\widehat{\vec p}=(\widehat{\rho},\widehat{v}_{\sh}, \widehat{\tau}_{\sh}, \widehat{v}_{\pr},\widehat{\tau}_{\pr})\in L^\infty(D)^5$ is given by 
$$
F''(V(\vec p))[V'(\vec p)\widehat{\vec p},V'(\vec p)\cdot]^*\vec g= \vec p_1^*+\vec p_2^*
$$
for $\vec g =(\vec g_{-1},\vec g_0,\dotsc,\vec g_L)\in L^2\big([0,T],L^2(D,\RR^3)\times L^2(D,\RS)^{1+L}\big)$ 
where
$$
\vec p_1^* =
\begin{pmatrix}
\int_0^T \big(\partial_t {\vec v} \cdot {\vec z}  - \frac{1}{\rho}\,\vec\varepsilon({\vec  v}):( {\vec \eta}_0 +\vec\Sigma_\eta)\big)\d t \\[2mm]
 \frac{2}{v_\sh} \int_0^T\big(-\vec\varepsilon({\vec v}): ( {\vec \eta}_0 +\vec\Sigma_\eta)+\pi\trace (\vec\Sigma^v)\Div {\vec v}\big)  \d t\\[2mm]
\frac{1}{1+\alpha\tau_\sh}\int_0^T \big( \vec\varepsilon({\vec v}):\vec\Sigma^\tau_{\sh,2} +\pi\trace( \vec\Sigma^\tau_{\sh,1})\Div {\vec v}\big) \d t
\\[2mm]
-\frac{2\pi}{v_\pr}\int_0^T \trace(\vec\Sigma^v) \Div {\vec  v}\, \d t
\\[2mm]
\frac{\pi}{1+\alpha\tau_\pr}\int_0^T \trace(\vec\Sigma^\tau_\pr) \Div {\vec v}\, \d t
\end{pmatrix}\! \!\in\! L^1(D)^5
$$
and 
$$
\vec p_2^*= \begin{pmatrix}
\int_0^T \big(\partial_t {\vec v}_1 \cdot {\vec w}  - \frac{1}{\rho}\big[\vec\varepsilon(\vec v_1):(\vec\varphi_0+\vec\Sigma_\varphi) +
\vec\varepsilon(\vec v) :\vec\Gamma_1^\rho+\trace(\vec\Gamma_2^\rho)\Div\V\big]\big)\d t \\[2mm]
 \frac{2}{v_\sh} \int_0^T\big(-\vec\varepsilon(\vec v_1):(\vec\varphi_0+\vec\Sigma_\varphi) 
 -\vec\varepsilon(\vec v) :\vec\Gamma_{1}^\rho \\
 \qquad\qquad\qquad\qquad\qquad\qquad+\trace(\vec\Gamma^v_{\sh,1})\Div\V_1+\trace(\vec\Gamma^v_{\sh,2})\Div\V
 \big)  \d t\\[2mm]
\frac{1}{1+\alpha\tau_\sh}
\int_0^T\big(-\vec\varepsilon(\vec v_1):(\vec\varphi_0+\frac{1}{\tau_\sh}\vec\Sigma_\varphi) 
 -\vec\varepsilon(\vec v) :\vec\Gamma_{\sh,0}^\tau \\
 \qquad\qquad\qquad\qquad\qquad\qquad+\trace(\vec\Gamma^\tau_{\sh,1})\Div\V_1+\trace(\vec\Gamma^\tau_{\sh,2})\Div\V\big)  \d t\\[2mm]
-\frac{2\pi}{v_\pr}\int_0^T \big(\trace(\vec\Gamma^v_{\pr,1})\Div\V_1+\trace(\vec\Gamma^v_{\pr,2}) \Div\V\big) \d t
\\[2mm]
\frac{\pi}{1+\alpha\tau_\pr}\int_0^T  \big(\trace(\vec\Gamma^\tau_{\pr,1})\Div\V_1-\trace(\vec\Gamma^\tau_{\pr,2}) \Div\V\big) \d t
\end{pmatrix}\! \!\in\! L^1(D)^5.
$$
Above, ${\vec v}$ is the first component of the solution of \eqref{C2:elawave_trans}, 
$z=({\vec z},  {\vec \eta}_0,\dotsc, {\vec \eta}_L)$ solves \eqref{evolution_back} with  $g=-V'(\vec p)\widehat{\vec p}(w'+Qw)$ and $w=(\vec w,  {\vec \varphi}_0,\dotsc, {\vec \varphi}_L)$ solves \eqref{C2:elawave_trans_ad}.
Further, 
$\vec\Sigma_\eta=\sum_{l=1}^L\vec\eta_l$, $\vec\Sigma_\varphi=\sum_{l=1}^L\vec\varphi$, and  the quantities 
$\vec\Sigma^v$, $\vec\Sigma^\tau_{\sh,1}$, $\vec\Sigma^\tau_{\sh,2}$, and $\vec\Sigma^\tau_\pr$
are exactly those from Theorem~{\rm\ref{th:FWI_adjoint}}, however, with $\vec\eta_l$ in place of $\vec\varphi_l$.
Finally, $\vec v_1=(F'(V(\vec p)) V'(\vec p) \widehat{\vec p})_1$, that is,  $\vec v_1$ is the first component of the solution of \eqref{FWI_first}, and
\begin{align*}
\vec\Gamma^\rho_1 &= \Big(\frac{\widehat{\rho}}{\rho}+\frac{\widetilde{\mu}}{\mu}\Big)\vec\varphi_0
+ \Big(\frac{\widehat{\rho}}{\rho}+\frac{\widehat{\mu}}{\tau_\sh\mu}\Big)\vec\Sigma_\varphi,\quad
\vec\Gamma^\rho_2 = \frac{\mu\widetilde{\pi}-\widetilde{\mu}\pi}{\mu(3\pi-4\mu)}\,\vec\varphi_0
+\frac{\tau_\sh\mu\widehat{\pi}-\tau_\pr\widehat{\mu}\pi}{\tau_\sh\mu(3\tau_\pr\pi-4\tau_\sh\mu)}\,\vec\Sigma_\varphi,\\[1mm]
\vec\Gamma^v_{\sh,1}&= \frac{\pi}{3\pi-4\mu} \,\vec\varphi_0+\frac{\tau_\pr\pi}{3\tau_\pr\pi-4\tau_\sh\mu} \,\vec\Sigma_\varphi\\[1mm]
\vec\Gamma^v_{\sh,2}&= \Big(\frac{\widehat{\rho}}{\rho}\,\frac{\pi}{3\pi-4\mu}-\frac{3\widetilde{\mu}\pi^2-4\widetilde{\pi}\mu^2}{\mu(3\pi-4\mu)^2}\Big)\vec\varphi_0\\
&\qquad\qquad\qquad\qquad\qquad\qquad+ \Big(\frac{\widehat{\rho}}{\rho}\,\frac{\tau_\pr\pi}
{3\tau_\pr\pi-4\tau_\sh\mu}-\frac{3\widehat{\mu}\tau_\pr^2\pi^2-4\widehat{\pi}\tau_\sh^2\mu^2}{\tau_\sh\mu(3\tau_\pr\pi-4\tau_\sh\mu)^2}\Big)\vec\Sigma_\varphi,
\\[1mm]
\vec\Gamma^\tau_{\sh,0}&= -\alpha\Big(\frac{\widehat{\rho}}{\rho}+\frac{\widetilde{\mu}}{\mu}\Big) \vec\varphi_0
+  \Big(\frac{\widehat{\rho}}{\tau_\sh\rho}+\frac{\widehat{\mu}}{\tau_\sh^2\mu}\Big)\vec\Sigma_\varphi,\\[1mm]
\vec\Gamma^\tau_{\sh,1}&= -\alpha\frac{\pi}{3\pi-4\mu} \,\vec\varphi_0+\frac{\tau_\pr\pi}{\tau_\sh(3\tau_\pr\pi-4\tau_\sh\mu)} \,\vec\Sigma_\varphi,
\\[1mm]
\vec\Gamma^\tau_{\sh,2}&= -\alpha\Big(\frac{\widehat{\rho}}{\rho}\,\frac{\pi}{3\pi-4\mu}-\frac{3\widetilde{\mu}\pi^2-4\widetilde{\pi}\mu^2}{\mu(3\pi-4\mu)^2}\Big)\vec\varphi_0\\
&\qquad\qquad\qquad\qquad\qquad\qquad+ \Big(\frac{\widehat{\rho}}{\rho}\,\frac{\tau_\pr\pi}
{\tau_\sh(3\tau_\pr\pi-4\tau_\sh\mu)}-\frac{3\widehat{\mu}\tau_\pr^2\pi^2-4\widehat{\pi}\tau_\sh^2\mu^2}{\tau_\sh^2\mu(3\tau_\pr\pi-4\tau_\sh\mu)^2}\Big)\vec\Sigma_\varphi,
\\[1mm]
\vec\Gamma^v_{\pr,1}&= \frac{1}{3\pi-4\mu} \,\vec\varphi_0+\frac{\tau_\pr}{3\tau_\pr\pi-4\tau_\sh\mu} \,\vec\Sigma_\varphi,
\\[1mm]
\vec\Gamma^v_{\pr,2}&= \Big(\frac{\widehat{\rho}}{\rho}\,\frac{1}{3\pi-4\mu}+
\frac{3\widetilde{\pi}\pi^2-4\widetilde{\mu}\mu^2}{\mu^2(3\pi-4\mu)^2}\Big)\vec\varphi_0\\
&\qquad\qquad\qquad\qquad\qquad\qquad+\tau_\pr\Big(\frac{\widehat{\rho}}{\rho}\,\frac{1}{3\tau_\pr\pi-4\tau_\sh\mu}+\frac{3\widehat{\pi}\tau_\pr^2\pi^2-4\widehat{\mu}\tau_\sh\mu^2}{\tau_\sh^2\mu^2(3\tau_\pr\pi-4\tau_\sh\mu)^2}
\Big)\vec\Sigma_\varphi,\\[1mm]
\vec\Gamma^\tau_{\pr,1}&= \frac{\alpha}{3\pi-4\mu} \,\vec\varphi_0-\frac{1}{3\tau_\pr\pi-4\tau_\sh\mu} \,\vec\Sigma_\varphi,
\\[1mm]
\vec\Gamma^\tau_{\pr,2}&= \alpha\Big(\frac{\widehat{\rho}}{\rho}\,\frac{1}{3\pi-4\mu}+
\frac{3\widetilde{\pi}\pi^2-4\widetilde{\mu}\mu^2}{\mu^2(3\pi-4\mu)^2}\Big)\vec\varphi_0\\
&\qquad\qquad\qquad\qquad\qquad\qquad+\Big(\frac{\widehat{\rho}}{\rho}\,\frac{1}{3\tau_\pr\pi-4\tau_\sh\mu}+\frac{3\widehat{\pi}\tau_\pr^2\pi^2-4\widehat{\mu}\tau_\sh\mu^2}{\tau_\sh^2\mu^2(3\tau_\pr\pi-4\tau_\sh\mu)^2}
\Big)\vec\Sigma_\varphi.
\end{align*}
\end{theorem}
\begin{proof}
From the previous theorem we have that
\begin{align*}
\big(F''(V(\vec p))[V'(\vec p)\widehat{\vec p}_1,V'(\vec p)\,\cdot]^*\vec g\big)\widehat{\vec p}_2 &=
 \int_0^T \langle  V'(\vec p)\widehat{\vec p}_1({u}_2^\prime(t)+Q{u}_2(t)),w(t)\big\rangle_X\d t\\
&\qquad\qquad+\int_0^T \langle  V'(\vec p)\widehat{\vec p}_2({u}_1^\prime(t)+Q{u}_1(t)),w(t)\big\rangle_X\d t. 
\end{align*}
Note that here ${u}_i=F'(V(\vec p)) V'(\vec p) \widehat{\vec p}_i $ solves \eqref{FWI_first} with $\widehat{\vec p}=\widehat{\vec p}_i$ and $w$ solves \eqref{C2:elawave_trans_ad}. 

We start with 
\begin{align*}
\int_0^T \langle  V'(\vec p)\widehat{\vec p}_1({u}_2^\prime(t)+Q{u}_2(t)),w(t)\big\rangle_X\d t &= -
\int_0^T \langle u_2,  V'(\vec p)\widehat{\vec p}_1(w^\prime(t)+Qw(t))\big\rangle_X\d t\\
&= \int_0^T \langle V'(\vec p)\widehat{\vec p}_2({u}^\prime(t)+Q{u}(t)),  z(t)\big\rangle_X\d t.
\end{align*}
We are now exactly in the situation of the proof of Theorem~\ref{th:FWI_adjoint}, see \eqref{eq:aux_FWI_ad},
and proceed accordingly which yields the representation of $\vec p_1^*$.

Similar to \eqref{sum} we find
for $\widehat{\vec p}_2=(\widehat{\rho}_2, \widehat{v}_{\sh,2}, \widehat{\tau}_{\sh,2}, \widehat{v}_{\pr,2}, \widehat{\tau}_{\pr,2})$ that
\begin{equation*} 
\big\langle V'(\vec p)\widehat{\vec p}_2\big(u_1' +Qu_1\big),w\rangle_X =\int_D\big(\widehat{\rho}_2\, 
\partial_t\vec v_1\cdot {\vec w} 
+S_0+S_1+\cdots + S_L\big)\d x
\end{equation*}
with 
$$
S_0= \Big[-\frac{\widehat{\rho}_2}{\rho^2}\widetilde{C}(\mu,\pi)\partial_t\Sig_{0,1} -\frac{1}{\rho} \widetilde{C}(\mu,\pi)
C(\widetilde{\mu}_2,\widetilde{\pi}_2) \widetilde{C}(\mu,\pi) \partial_t\Sig_{0,1}\Big]: {\vec \varphi}_0
$$
and, for $l=1,\dotsc,L$, 
\begin{multline*}
S_l=\Big[-\frac{\widehat{\rho}_2}{\rho^2}\widetilde{C}(\tau_\sh\mu,\tau_\pr\pi)\Big(\partial_t\Sig_{l,1} +\frac{\Sig_{l,1}}{\tau_{\Sig,l}}\Big)\\
-\frac{1}{\rho} \widetilde{C}(\tau_\sh\mu,\tau_\pr\pi)
C(\widehat{\mu}_2,\widehat{\pi}_2) \widetilde{C}(\tau_\sh\mu,\tau_\pr\pi) \Big(\partial_t\Sig_{l,1} +\frac{\Sig_{l,1}}{\tau_{\Sig,l}}\Big)\Big]: {\vec \varphi}_l.
\end{multline*}
In view of \eqref{FWI_first2} and \eqref{FWI_first3} we have that
$$
\partial_t {\vec \sigma}_{0,1} = C(\mu_0,\pi_0)\vec \varepsilon({\vec v}_1 ) + \big(\widehat{\rho}_1\,
C(\mu,\pi)+\rho\,C(\widetilde{\mu}_1,\widetilde{\pi}_1)\big) \vec\varepsilon(\vec v)
$$
and
$$
\partial_t\Sig_{l,1} +\frac{\Sig_{l,1}}{\tau_{\Sig,l}} = 
C(\tau_\sh\mu_0,\tau_\pr\pi_0)\vec \varepsilon({\vec v}_1)  
 + \big(\widehat{\rho}_1\,
C(\tau_\sh\mu,\tau_\pr\pi)+\rho\,C(\widehat{\mu}_1\widehat{\pi}_1)\big)\vec \varepsilon(\vec v),\quad l=1,\dotsc,L,
$$
which we plug into $S_0$ and $S_l$, respectively. Thus,
\begin{multline*}
S_0=-\frac{\widehat{\rho}_2}{\rho} \vec\varepsilon(\vec v_1):\vec\varphi_0 
- \frac{\widehat{\rho}_1\widehat{\rho}_2}{\rho^2} \vec\varepsilon(\vec v):\vec\varphi_0 
- C(\widetilde{\mu}_2,\widetilde{\pi}_2) 
\vec\varepsilon(\vec v_1):\widetilde{C}(\mu,\pi)\vec\varphi_0
\\
- \frac{\widehat{\rho}_2}{\rho}  C(\widetilde{\mu}_1,\widetilde{\pi}_1) 
\vec\varepsilon(\vec v):\widetilde{C}(\mu,\pi)\vec\varphi_0
-\frac{\widehat{\rho}_1}{\rho}  C(\widetilde{\mu}_2,\widetilde{\pi}_2) 
\vec\varepsilon(\vec v):\widetilde{C}(\mu,\pi)\vec\varphi_0 \\
-\widetilde{C}(\mu,\pi)C(\widetilde{\mu}_1,\widetilde{\pi}_1) 
\vec\varepsilon(\vec v): C(\widetilde{\mu}_2,\widetilde{\pi}_2) \widetilde{C}(\mu,\pi)\vec\varphi_0
\end{multline*}
and
\begin{multline*}
S_l =-\frac{\widehat{\rho}_2}{\rho}\vec \varepsilon(\vec v_1):\vec\varphi_l
- \frac{\widehat{\rho}_1\widehat{\rho}_2}{\rho^2}\vec \varepsilon(\vec v):\vec\varphi_l
- C(\widehat{\mu}_2,\widehat{\pi}_2) 
\vec\varepsilon(\vec v_1):\widetilde{C}(\tau_\sh\mu,\tau_\pr\pi)\vec\varphi_l
\\
- \frac{\widehat{\rho}_2}{\rho}  C(\widehat{\mu}_1,\widehat{\pi}_1) 
\vec\varepsilon(\vec v):\widetilde{C}(\tau_\sh\mu,\tau_\pr\pi)\vec\varphi_l
-\frac{\widehat{\rho}_1}{\rho}  C(\widehat{\mu}_2,\widehat{\pi}_2) 
\vec\varepsilon(\vec v):\widetilde{C}(\tau_\sh\mu,\tau_\pr\pi)\vec\varphi_l \\
-\widetilde{C}(\tau_\sh\mu,\tau_\pr\pi)C(\widehat{\mu}_1,\widehat{\pi}_1) 
\vec\varepsilon(\vec v): C(\widehat{\mu}_2,\widehat{\pi}_2) \widetilde{C}(\tau_\sh\mu,\tau_\pr\pi)\vec\varphi_l.
\end{multline*}
We have
$$
\widetilde{C}(\mu,\pi)C(\widetilde{\mu}_1,\widetilde{\pi}_1)\vec\varepsilon(\V) = \frac{\widetilde{\mu}_1}{\mu} 
\vec\varepsilon(\V)  +\frac{\mu \widetilde{\pi}_1-\widetilde{\mu}_1\pi}{\mu(3\pi-4\mu)}\,\Div\V\, 
\vec I
$$
and 
$$
C(\widetilde{\mu}_2,\widetilde{\pi}_2)\widetilde{C}(\mu,\pi)\vec \varphi_0 =\frac{\widetilde{\mu}_2}{\mu} 
\vec\varphi_0  +\frac{\mu \widetilde{\pi}_2-\widetilde{\mu}_2\pi}{\mu(3\pi-4\mu)}\,\trace(\vec\varphi_0)\,\vec I
$$
so that
\begin{multline*}
\widetilde{C}(\mu,\pi)C(\widetilde{\mu}_1,\widetilde{\pi}_1) 
\vec\varepsilon(\vec v): C(\widetilde{\mu}_2,\widetilde{\pi}_2) \widetilde{C}(\mu,\pi)\vec\varphi_0 = \frac{\widetilde{\mu}_1\widetilde{\mu}_2}{\mu^2}\, \vec\varepsilon(\V): \vec\varphi_0 \\
+ \frac{\widetilde{\mu}_2(3\widetilde{\mu}_1\pi^2-4\widetilde{\pi}_1\mu^2)+\widetilde{\pi}_2
(3\widetilde{\pi}_1\pi^2-4\widetilde{\mu}_1\mu^2)}{\mu^2(3\pi-4\mu)^2}
\,\Div\V\, \trace(\vec\varphi_0) .
\end{multline*}
Plugging these identities into $S_0$ (and an adapted version into $S_l$) and recalling \eqref{eq:aux_FWI_ad2} we find that
\begin{align*}
S_0&= \widehat{\rho}_2\Bigg( -\frac{1}{\rho} \Big[\vec\varepsilon(\vec v_1)+\Big(\frac{\widehat{\rho}_1}{\rho}
+\frac{\widetilde{\mu}_1}{\mu}\Big)\vec\varepsilon(\vec v)\Big]:\vec\varphi_0 +\frac{1}{\rho}\Big(
\frac{\widetilde{\mu}_1\pi}{\mu(3\pi-4\mu)}-
\frac{\widetilde{\pi}_1}{3\pi-4\mu}\Big)\Div\V \trace(\vec\varphi_0)\Bigg)\\[1mm]
&\quad+\widetilde{\mu}_2 \Bigg(-\frac{1}{\mu}\vec\varepsilon(\V_1): \vec\varphi_0
 -\Big(\frac{\widehat{\rho}_1}{\rho\mu}+\frac{\widetilde{\mu}_1}{\mu^2}\Big) \vec\varepsilon(\V): \vec\varphi_0
+\frac{\pi}{\mu(3\pi-4\mu)} \Div\V_1\, \trace(\vec\varphi_0) 
\\
 &\qquad \qquad \qquad \qquad  \qquad \qquad-\Big(\frac{\widehat{\rho}_1}{\rho}\,\frac{\pi}{\mu(3\pi-4\mu)}+\frac{3\widetilde{\mu}_1\pi^2-4\widetilde{\pi}_1\mu^2}{\mu^2(3\pi-4\mu)^2}\Big) \Div\V\, \trace(\vec\varphi_0) \Bigg)\\[1mm]
&\quad+\widetilde{\pi}_2\Bigg(\frac{-1}{3\pi-4\mu}\Div\V_1\, \trace(\vec\varphi_0)
-\Big(\frac{\widehat{\rho}_1}{\rho}\,\frac{1}{3\pi-4\mu}-
\frac{3\widetilde{\pi}_1\pi^2-4\widetilde{\mu}_1\mu^2}{\mu^2(3\pi-4\mu)^2}\Big)\Div\V\, \trace(\vec\varphi_0)\Bigg)
\end{align*}
and
\begin{align*}
S_l&= \widehat{\rho}_2\Bigg( -\frac{1}{\rho} \Big[\vec\varepsilon(\vec v_1)+\Big(\frac{\widehat{\rho}_1}{\rho}
+\frac{\widehat{\mu}_1}{\tau_\sh\mu}\Big)\vec\varepsilon(\vec v)\Big]:\vec\varphi_l\\
 &\qquad \qquad \qquad \qquad  \qquad \qquad+\frac{1}{\rho}\Big(
\frac{\widehat{\mu}_1\tau_\pr\pi}{\tau_\sh\mu(3\tau_\pr\pi-4\tau_\sh\mu)}-
\frac{\widehat{\pi}_1}{3\tau_\pr\pi-4\tau_\sh\mu}\Big)\Div\V \trace(\vec\varphi_l)\Bigg)\\[1mm]
&\quad+\widehat{\mu}_2 \Bigg(-\frac{1}{\tau_\sh\mu}\vec\varepsilon(\V_1): \vec\varphi_l
-\Big(\frac{\widehat{\rho}_1}{\rho\tau_\sh\mu}+\frac{\widehat{\mu}_1}{\tau_\sh^2\mu^2}\Big) \vec\varepsilon(\V): \vec\varphi_l
+\frac{\tau_\pr\pi}{\tau_\sh\mu(3\tau_\pr\pi-4\tau_\sh\mu)} 
\Div\V_1\, \trace(\vec\varphi_l) 
 \\
 &\qquad \qquad \qquad \qquad  \qquad -\Big(\frac{\widehat{\rho}_1}{\rho}\,\frac{\tau_\pr\pi}{\tau_\sh\mu(3\tau_\pr\pi-4\tau_\sh\mu)}+\frac{3\widehat{\mu}_1\tau_\pr^2\pi^2-4\widehat{\pi}_1\tau_\sh^2\mu^2}{\tau_\sh^2\mu^2(3\tau_\pr\pi-4\tau_\sh\mu)^2}\Big) \Div\V\, \trace(\vec\varphi_l) \Bigg)\\[1mm]
&\quad+\widehat{\pi}_2\Bigg(\frac{-1}{3\tau_\pr\pi-4\tau_\sh\mu}\Div\V_1\, \trace(\vec\varphi_l)\\
&\qquad \qquad \qquad \qquad \qquad  -\Big(\frac{\widehat{\rho}_1}{\rho}\,\frac{1}{3\tau_\pr\pi-4\tau_\sh\mu}-
\frac{3\widehat{\pi}_1\tau_\pr^2\pi^2-4\widehat{\mu}_1\tau_\sh^2\mu^2}{\tau_\sh^2\mu^2(3\tau_\pr\pi-4\tau_\sh\mu)^2}\Big)\Div\V\, \trace(\vec\varphi_l)\Bigg).
\end{align*}
Next we replace $\widetilde{\mu}_2$, $\widetilde{\pi}_2$, and $\widehat{\mu}_2$,  $\widehat{\pi}_2$  by their values
from \eqref{eq:aux_FWI_ad3} and \eqref{eq:aux_FWI_ad4}, respectively. Finally, we calculate $S_0+\cdots+S_L$ and group the terms 
belonging to the components of~$\widehat{\vec p}_2$.
\end{proof}
\begin{theorem}\label{FWI''_adjoint2}
Under the assumptions of Theorem~{\rm\ref{th:FWI_second}} we have that
the adjoint $$F'(V(\vec p))V''(\vec p)[\widehat{\vec p},\cdot]^*\in \mathcal{L}\big( L^2([0,T],X),( L^\infty(D)^5)'\big)$$ at $\vec p=(\rho,v_{\sh}, \tau_{\sh}, v_{\pr},\tau_{\pr})\in \mathsf{D}(\Phi)$ and $\widehat{\vec p}=(\widehat{\rho},\widehat{v}_{\sh}, \widehat{\tau}_{\sh}, \widehat{v}_{\pr},\widehat{\tau}_{\pr})\in L^\infty(D)^5$ is given by 
$$
F'(V(\vec p))V''(\vec p)[\widehat{\vec p},\cdot]^*\vec g = 
\begin{pmatrix}
\frac{1}{\rho}\int_0^T \big(\vec\varepsilon(\vec v): \vec\Upsilon^\rho_1 + \trace(\vec\Upsilon^\rho_2)\Div \V\big)\d t \\[2mm]
\frac{2}{v_\sh} \int_0^T \big(\vec\varepsilon(\vec v): \vec\Upsilon^v_{\sh,1}+  \trace(\vec\Upsilon^v_{\sh,2})\Div \V\big) \d t\\[2mm]
\frac{1}{1+\alpha\tau_\sh} \int_0^T \big(\vec\varepsilon(\vec v): \vec\Upsilon^\tau_{\sh,1}+  \trace(\vec\Upsilon^\tau_{\sh,2})\Div \V\big)\d t
\\[2mm]
\frac{2\pi}{v_\pr}\int_0^T\trace(\vec\Upsilon^v_{\pr})\Div \V\, \d t
\\[2mm]
\frac{\pi}{1+\alpha\tau_\pr}\int_0^T \trace(\vec\Upsilon^\tau_{\pr})\Div \V\,  \d t
\end{pmatrix} \in L^1(D)^5
$$
for $\vec g =(\vec g_{-1},\vec g_0,\dotsc,\vec g_L)\in L^2\big([0,T],L^2(D,\RR^3)\times L^2(D,\RS)^{1+L}\big)$ where
 $\vec v$ is the first component of the solution of \eqref{C2:elawave_trans}. Let
$w=({\vec w},  {\vec \varphi}_0,\dotsc, {\vec \varphi}_L)$ solve \eqref{C2:elawave_trans_ad} with $w(T)=0$ and set  $\vec\Sigma=\sum_{l=1}^L\vec\varphi_l$. Then,
\begin{align*}
\vec\Upsilon^\rho_1 &= \Big(\frac{\widehat{\rho}}{\rho}+\frac{\widetilde{\mu}}{\mu}\Big)\vec\varphi_0
+ \Big(\frac{\widehat{\rho}}{\rho}+\frac{\widehat{\mu}}{\tau_\sh\mu}\Big)\vec\Sigma,\quad
\vec\Upsilon^\rho_2 = \frac{\mu\widetilde{\pi}-\widetilde{\mu}\pi}{\mu(3\pi-4\mu)}\,\vec\varphi_0
+\frac{\tau_\sh\mu\widehat{\pi}-\tau_\pr\widehat{\mu}\pi}{\tau_\sh\mu(3\tau_\pr\pi-4\tau_\sh\mu)}\,\vec\Sigma,\\[1mm]
\vec\Upsilon^v_{\sh,1}&= \Big(\frac{\widehat{\rho}}{\rho}+\frac{2\widetilde{\mu}}{\mu}\Big) \vec\varphi_0
+  \Big(\frac{\widehat{\rho}}{\rho}+\frac{2\widehat{\mu}}{\tau_\sh\mu}\Big)\vec\Sigma,\\[1mm]
\vec\Upsilon^v_{\sh,2}&= \Big(2\,\frac{3\widetilde{\mu}\pi^2-4\widetilde{\pi}\mu^2}{\mu(3\pi-4\mu)^2}-\frac{\widehat{\rho}}{\rho}\,\frac{\pi}{3\pi-4\mu}\Big)\vec\varphi_0\\
&\qquad\qquad\qquad\qquad\qquad\qquad+ \Big(2\,\frac{3\widehat{\mu}\tau_\pr^2\pi^2-4\widehat{\pi}\tau_\sh^2\mu^2}{\tau_\sh\mu(3\tau_\pr\pi-4\tau_\sh\mu)^2}-\frac{\widehat{\rho}}{\rho}\,\frac{\tau_\pr\pi}
{3\tau_\pr\pi-4\tau_\sh\mu}\Big)\vec\Sigma,
\\[1mm]
\vec\Upsilon^\tau_{\sh,1}&= -\alpha\Big(\frac{\widehat{\rho}}{\rho}+\frac{2\widetilde{\mu}}{\mu}\Big) \vec\varphi_0
+  \Big(\frac{\widehat{\rho}}{\tau_\sh\rho}+\frac{2\widehat{\mu}}{\tau_\sh^2\mu}\Big)\vec\Sigma,\\[1mm]
\vec\Upsilon^\tau_{\sh,2}&= -\alpha\Big(2\,\frac{3\widetilde{\mu}\pi^2-4\widetilde{\pi}\mu^2}{\mu(3\pi-4\mu)^2}-\frac{\widehat{\rho}}{\rho}\,\frac{\pi}{3\pi-4\mu}\Big)\vec\varphi_0\\
&\qquad\qquad\qquad\qquad\qquad\qquad+ \Big(2\,\frac{3\widehat{\mu}\tau_\pr^2\pi^2-4\widehat{\pi}\tau_\sh^2\mu^2}{\tau_\sh^2\mu(3\tau_\pr\pi-4\tau_\sh\mu)^2}-\frac{\widehat{\rho}}{\rho}\,\frac{\tau_\pr\pi}
{\tau_\sh(3\tau_\pr\pi-4\tau_\sh\mu)}\Big)\vec\Sigma,
\\[1mm]
\vec\Upsilon^v_{\pr}&= \Big(\frac{\widehat{\rho}}{\rho}\,\frac{1}{3\pi-4\mu}+2\,
\frac{3\widetilde{\pi}\pi^2-4\widetilde{\mu}\mu^2}{\mu^2(3\pi-4\mu)^2}\Big)\vec\varphi_0\\
&\qquad\qquad\qquad\qquad\qquad\qquad+\tau_\pr\Big(\frac{\widehat{\rho}}{\rho}\,\frac{1}{3\tau_\pr\pi-4\tau_\sh\mu}+2\,
\frac{3\widehat{\pi}\tau_\pr^2\pi^2-4\widehat{\mu}\tau_\sh^2\mu^2}{\tau_\sh^2\mu^2(3\tau_\pr\pi-4\tau_\sh\mu)^2}\Big)\vec\Sigma,
\\[1mm]
\vec\Upsilon^\tau_{\pr}&= -\alpha\Big(\frac{\widehat{\rho}}{\rho}\,\frac{1}{3\pi-4\mu}+2\,
\frac{3\widetilde{\pi}\pi^2-4\widetilde{\mu} \mu^2}{\mu^2(3\pi-4\mu)^2}\Big)\vec\varphi_0\\
&\qquad\qquad\qquad\qquad\qquad\qquad+\Big(\frac{\widehat{\rho}}{\rho}\,\frac{1}{3\tau_\pr\pi-4\tau_\sh\mu}+2\,
\frac{3\widehat{\pi}\tau_\pr^2\pi^2-4\widehat{\mu}\tau_\sh^2\mu^2}{\tau_\sh^2\mu^2(3\tau_\pr\pi-4\tau_\sh\mu)^2}
\Big)\vec\Sigma,
\end{align*}
with the abbreviations $\widetilde{\mu}$, $\widetilde{\pi}$, and $\widehat{\mu}$,  $\widehat{\pi}$  from \eqref{eq:aux_FWI_ad3} and \eqref{eq:aux_FWI_ad4} which depend on $\widehat{\vec p}$.
\end{theorem}
\begin{proof}
Since 
$$
\big(F'(V(\vec p))V''(\vec p)[\widehat{\vec p}_1,\cdot]^*\vec g\big)\widehat{\vec p}_2 \stackrel{\eqref{eq:aux_FWI_ad}}{=}
\int_0^T\! \big\langle V''(\vec p)[\widehat{\vec p}_1,\widehat{\vec p}_2]\big(u'(t) +Qu(t)\big),w(t)\rangle_X\,\d t.
$$
we are basically again in the situation of the  proof of Theorem~\ref{th:FWI_adjoint}.  Using \eqref{V''}  we find that
$$
\big\langle V''(\vec p)[\widehat{\vec p}_1,\widehat{\vec p}_2]\big(u' +Qu\big),w\rangle_X = 
\int_D\big(S_0+S_1+\cdots + S_L\big)\d x
$$
with
\begin{multline*}
S_0=\Bigg(\frac{\widehat{\rho}_1\widehat{\rho}_2}{\rho^3}\widetilde{C}(\mu,\pi)-
\frac{\widehat{\rho}_1}{\rho^2}\widetilde{C}'(\mu,\pi)\begin{bmatrix}
\widetilde{\mu}_2\\[1mm]
\widetilde{\pi}_2
\end{bmatrix} -\frac{\widehat{\rho}_2}{\rho^2} \widetilde{C}'(\mu,\pi)\!
\begin{bmatrix}
\widetilde{\mu}_1\\[1mm]
\widetilde{\pi}_1
\end{bmatrix}\\ +\frac{1}{\rho} \widetilde{C}''(\mu,\pi)\!
\begin{bmatrix}
\widetilde{\mu}_1\\[1mm]
\widetilde{\pi}_1
\end{bmatrix}
\begin{bmatrix}
\widetilde{\mu}_2\\[1mm]
\widetilde{\pi}_2
\end{bmatrix}\Bigg)\partial_t\Sig_0:\vec \varphi_0
\end{multline*}
and
\begin{multline*}
S_l= \Bigg(\frac{\widehat{\rho}_1\widehat{\rho}_2}{\rho^3} \widetilde{C}(\tau_\sh\mu,\tau_\pr\pi)
-\frac{\widehat{\rho}_1}{\rho^2}\widetilde{C}'(\tau_\sh\mu,\tau_\pr\pi)\begin{bmatrix}
\widehat{\mu}_2\\[1mm]
\widehat{\pi}_2
\end{bmatrix} -\frac{\widehat{\rho}_2}{\rho^2} \widetilde{C}'(\tau_\sh\mu,\tau_\pr\pi)\!
\begin{bmatrix}
\widehat{\mu}_1\\[1mm]
\widehat{\pi}_1
\end{bmatrix}\\+\frac{1}{\rho} \widetilde{C}''(\tau_\sh\mu,\tau_\pr\pi)\!
\begin{bmatrix}
\widehat{\mu}_1\\[1mm]
\widehat{\pi}_1
\end{bmatrix}
\begin{bmatrix}
\widehat{\mu}_2\\[1mm]
\widehat{\pi}_2
\end{bmatrix}\Bigg)\Big(\partial_t \Sig_l+\frac{\Sig_l}{\tau_{\Sig,l}}\Big):\vec\psi_l,\quad l=1,\dotsc,L.
\end{multline*}
First we simplify $S_0$. By \eqref{C2:elawave2_trans},
$$
\frac{1}{\rho} \widetilde{C}(\mu,\pi)\partial_t\Sig_0:\vec\varphi_0=
 \vec\varepsilon(\V): \vec\varphi_0.
$$
Further, in view of \eqref{eq:aux_FWI_ad2},
\begin{multline*}
-\frac{1}{\rho}\widetilde{C}'(\mu,\pi)\begin{bmatrix}
\widetilde{\mu}_i\\[1mm]
\widetilde{\pi}_i
\end{bmatrix}\partial_t\Sig_0:\vec\varphi_0\\
= \widetilde{\mu}_i\Big(\frac{1}{\mu } \vec\varepsilon(\vec v):{\vec \varphi}_0
- \frac{\pi}{\mu(3\pi-4\mu)}\,\Div\V\, \trace({\vec \varphi}_0)\Big)+\frac{\widetilde{\pi}_i}{3\pi-4\mu} 
\,\Div\V\, \trace({\vec \varphi}_0),\quad i=1,2.
\end{multline*}
Next, using \eqref{C2:elawave2_trans} and \eqref{C''}  we get
\begin{multline*}
\frac{1}{\rho} \widetilde{C}''(\mu,\pi)\!
\begin{bmatrix}
\widetilde{\mu}_1\\[1mm]
\widetilde{\pi}_1
\end{bmatrix}
\begin{bmatrix}
\widetilde{\mu}_2\\[1mm]
\widetilde{\pi}_2
\end{bmatrix}\partial_t\Sig_0:\vec \varphi_0 = \widetilde{C}(\mu,\pi)C(\widetilde{\mu}_1,\widetilde{\pi}_1)\vec\varepsilon(\V)
:C(\widetilde{\mu}_2,\widetilde{\pi}_2)\widetilde{C}(\mu,\pi)\vec \varphi_0\\
+ \widetilde{C}(\mu,\pi)C(\widetilde{\mu}_2,\widetilde{\pi}_2)\vec\varepsilon(\V)
:C(\widetilde{\mu}_1,\widetilde{\pi}_1)\widetilde{C}(\mu,\pi)\vec \varphi_0.
\end{multline*}
We have
$$
\widetilde{C}(\mu,\pi)C(\widetilde{\mu}_2,\widetilde{\pi}_2)\vec\varepsilon(\V) = \frac{\widetilde{\mu}_2}{\mu} 
\vec\varepsilon(\V)  +\frac{\mu \widetilde{\pi}_2-\widetilde{\mu}_2\pi}{\mu(3\pi-4\mu)}\,\Div\V\, 
\vec I
$$
and 
$$
C(\widetilde{\mu}_1,\widetilde{\pi}_1)\widetilde{C}(\mu,\pi)\vec \varphi_0 =\frac{\widetilde{\mu}_1}{\mu} 
\vec\varphi_0  +\frac{\mu \widetilde{\pi}_1-\widetilde{\mu}_1\pi}{\mu(3\pi-4\mu)}\,\trace(\vec\varphi_0)\,\vec I
$$
so that
\begin{multline*}
\frac{1}{\rho} \widetilde{C}''(\mu,\pi)\!
\begin{bmatrix}
\widetilde{\mu}_1\\[1mm]
\widetilde{\pi}_1
\end{bmatrix}
\begin{bmatrix}
\widetilde{\mu}_2\\[1mm]
\widetilde{\pi}_2
\end{bmatrix}\partial_t\Sig_0:\vec \varphi_0 = 2\,\frac{\widetilde{\mu}_1\widetilde{\mu}_2}{\mu^2}\, \vec\varepsilon(\V): \vec\varphi_0 \\
+ 2\,\frac{\widetilde{\mu}_2(3\widetilde{\mu}_1\pi^2-4\widetilde{\pi}_1\mu^2)+\widetilde{\pi}_2
(3\widetilde{\pi}_1\pi^2-4\widetilde{\mu}_1\mu^2)}{\mu^2(3\pi-4\mu)^2}
\,\Div\V\, \trace(\vec\varphi_0) .
\end{multline*}
Substituting above auxiliary results into the expression for $S_0$ yields 
\begin{align*}
S_0&= \widehat{\rho}_2\Bigg(\Big(\frac{\widehat{\rho}_1}{\rho^2}+\frac{\widetilde{\mu}_1}{\rho\mu}\Big) \vec\varepsilon(\V): \vec\varphi_0 +\frac{1}{\rho}\Big(\frac{\widetilde{\pi}_1}{3\pi-4\mu} -\frac{\widetilde{\mu}_1\pi}{\mu(3\pi-4\mu)}\Big) \,\Div\V\, \trace(\vec\varphi_0) \Bigg)\\
&\quad+\widetilde{\mu}_2 \Bigg(\Big(\frac{\widehat{\rho}_1}{\rho\mu}+\frac{2\widetilde{\mu}_1}{\mu^2}\Big) \vec\varepsilon(\V): \vec\varphi_0 +\Big(2\,\frac{3\widetilde{\mu}_1\pi^2-4\widetilde{\pi}_1\mu^2}{\mu^2(3\pi-4\mu)^2}-\frac{\widehat{\rho}_1}{\rho}\,\frac{\pi}{\mu(3\pi-4\mu)}\Big) \Div\V\, \trace(\vec\varphi_0) \Bigg)\\
&\quad+\widetilde{\pi}_2\Big(\frac{\widehat{\rho}_1}{\rho}\,\frac{1}{\mu(3\pi-4\mu)}+2\,
\frac{3\widetilde{\pi}_1-4\widetilde{\mu}_1}{(3\pi-4\mu)^2}\Big)\Div\V\, \trace(\vec\varphi_0).
\end{align*}
Similar computations for $l=1,\dotsc,L$ based on \eqref{C2:elawave3_trans} result in
\begin{align*}
S_l&= \widehat{\rho}_2\Bigg(\Big(\frac{\widehat{\rho}_1}{\rho^2}+\frac{\widehat{\mu}_1}{\rho\tau_\sh\mu}\Big) \vec\varepsilon(\V): \vec\varphi_l +\frac{1}{\rho}\Big(\frac{\widehat{\pi}_1}{3\tau_\pr\pi-4\tau_\sh\mu} -\frac{\widehat{\mu}_1\tau_\pr\pi}
{\tau_\sh\mu(3\tau_\pr\pi-4\tau_\sh\mu)}\Big) \,\Div\V\, \trace(\vec\varphi_l) \Bigg)\\
&\quad+\widehat{\mu}_2 \Bigg(\Big(\frac{\widehat{\rho}_1}{\rho\tau_\sh\mu}+\frac{2\widehat{\mu}_1}{\tau_\sh^2\mu^2}\Big) \vec\varepsilon(\V): \vec\varphi_l \\
&\qquad\qquad\qquad\qquad +\Big(2\,\frac{3\widehat{\mu}_1\tau_\pr^2\pi^2-4\widehat{\pi}_1\tau_\sh^2\mu^2}{\tau_\sh^2\mu^2(3\tau_\pr\pi-4\tau_\sh\mu)^2}-\frac{\widehat{\rho}_1}{\rho}\,\frac{\tau_\pr\pi}
{\tau_\sh\mu(3\tau_\pr\pi-4\tau_\sh\mu)}\Big) \Div\V\, \trace(\vec\varphi_l) \Bigg)\\
&\quad+\widehat{\pi}_2\Big(\frac{\widehat{\rho}_1}{\rho}\,\frac{1}{\tau_\sh\mu(3\tau_\pr\pi-4\tau_\sh\mu)}+2\,
\frac{3\widehat{\pi}_1-4\widehat{\mu}_1}{(3\tau_\pr\pi-4\tau_\sh\mu)^2}\Big)\Div\V\, \trace(\vec\varphi_l).
\end{align*}
Next we replace $\widetilde{\mu}_2$, $\widetilde{\pi}_2$, and $\widehat{\mu}_2$,  $\widehat{\pi}_2$  by their values
from \eqref{eq:aux_FWI_ad3} and \eqref{eq:aux_FWI_ad4}, respectively. Finally, we calculate $S_0+\cdots+S_L$ and group the terms 
belonging to the components of~$\widehat{\vec p}_2$.
\end{proof}
In view of \eqref{Phi''_2} we have now derived  an analytic expression for $\Phi''(\vec p)[\widehat{\vec p},\cdot]^*$ in rather basic terms. 
\appendix
\section{Two spatial dimensions}
The expressions for the Fr\'echet derivatives and their adjoints provided in the main part of this paper cannot directly be applied to
the viscoelastic wave equation in two spatial dimensions. The differences to the 3D case which have to be taken into account are
$$
\trace(\vec I)=2 \quad\text{and}\quad  \widetilde{C}(m,p)\vec M=C^{-1}(m,p)\vec M= \frac{1}{2m}\,\vec M +\frac{2m-p}{4m(p-m)}\,\trace(\vec M)\vec I. 
$$
With these ingredients the derivatives and adjoints can be 
calculated exactly along the lines presented on the previous pages.  

In this appendix we provide 2D versions of 
Theorems~\ref{th:FWI_adjoint}, \ref{FWI''_adjoint1}, and~\ref{FWI''_adjoint2}.

\begin{theorem}[2D version of Theorem~\ref{th:FWI_adjoint}]~\\
The only quantities which have to be changed are $\vec\Sigma^v$, $\vec\Sigma^\tau_{\sh,1}$, 
 and $\vec\Sigma^\tau_{\pr}$. With
\begin{align*}
\vec\Sigma^v&=\frac{1}{2(\pi-\mu)}\, {\vec \varphi}_0 + 
\frac{\tau_\pr}{2(\tau_\pr\pi-\tau_\sh\mu)} \,\vec\Sigma,\\[1mm]
\vec\Sigma^\tau_{\sh,1}&=-\frac{\alpha}{2(\pi-\mu)}\, {\vec \varphi}_0 
+\frac{\tau_\pr}{2\,\tau_\sh(\tau_\pr\pi-\tau_\sh\mu)}\,\vec\Sigma,\\[1mm]
\vec\Sigma^\tau_\pr&=\frac{\alpha}{2(\pi-\mu)}\, {\vec \varphi}_0 
-\frac{1}{2(\tau_\pr\pi-\tau_\sh\mu)}\,\vec\Sigma,
\end{align*}
 the statement of Theorem~{\rm \ref{th:FWI_adjoint}} can be copied  without any further changes.
\end{theorem}
\begin{proof}
The only difference to the 3D proof concerns the computation of, compare \eqref{eq:aux_FWI_ad2},
\begin{align*}
&C(\widetilde{\mu},\widetilde{\pi}) \vec\varepsilon(\vec v):\widetilde{C}(\mu,\pi) {\vec \varphi}_0\\
&\qquad\quad= \big( 2\widetilde{\mu}\,  \vec\varepsilon(\vec v) + (\widetilde{\pi}-2\widetilde{\mu} )\Div\V\, \vec I\big):
\Big( \frac{1}{2\mu}\,{\vec \varphi}_0 + \frac{2\mu-\pi}{4\mu(\pi-\mu)}\,\trace({\vec \varphi}_0)\vec I\Big)\\
&\qquad\quad= \widetilde{\mu}\Big(\frac{1}{\mu } \vec\varepsilon(\vec v):{\vec \varphi}_0
- \frac{\pi}{2\mu(\pi-\mu)}\,\Div\V\, \trace({\vec \varphi}_0)\Big)+\frac{\widetilde{\pi}}{2(\pi-\mu)} 
\,\Div\V\, \trace({\vec \varphi}_0).
\end{align*}
\end{proof}
\begin{theorem}[2D version of Theorem~\ref{FWI''_adjoint1}]~\\ \label{2D_FWI''_adjoint1}
The expression for $\vec p^*_1$ in Theorem~{\rm \ref{FWI''_adjoint1}} remains correct for the {\rm 2D} case when the {\rm 2D} versions  of  $\vec\Sigma^v$, $\vec\Sigma^\tau_{\sh,1}$, 
 and $\vec\Sigma^\tau_{\pr}$ from the above theorem are taken (with $\vec\eta_l$ in place of $\vec\varphi_l$).
 In the expression for $\vec p^*_2$  we need to redefine the following quantities:
 \begin{align*}
\vec\Gamma^\rho_2 &= \frac{\mu\widetilde{\pi}-\widetilde{\mu}\pi}{2\mu(\pi-\mu)}\,\vec\varphi_0
+\frac{\tau_\sh\mu\widehat{\pi}-\tau_\pr\widehat{\mu}\pi}{\tau_\sh\mu(\tau_\pr\pi-\tau_\sh\mu)}\,\vec\Sigma_\varphi,
\\[1mm]
\vec\Gamma^v_{\sh,1}&= \frac{\pi}{2(\pi-\mu)} \,\vec\varphi_0+\frac{\tau_\pr\pi}{2(\tau_\pr\pi-\tau_\sh\mu)} \,\vec\Sigma_\varphi,
 \\[1mm]
\vec\Gamma^v_{\sh,2}&= \Big(\frac{\widehat{\rho}}{\rho}\,\frac{\pi}{2(\pi-\mu)}-\mathrm{K}_\mu\Big)\vec\varphi_0+ \Big(\frac{\widehat{\rho}}{\rho}\,\frac{\tau_\pr\pi}
{2(\tau_\pr\pi-\tau_\sh\mu)}-\mathrm{K}_{\mu,\tau}\Big)\vec\Sigma_\varphi,
 \\[1mm]
\vec\Gamma^\tau_{\sh,1}&= -\alpha\frac{\pi}{2(\pi-\mu)} \,\vec\varphi_0+\frac{\tau_\pr\pi}{2\tau_\sh(\tau_\pr\pi-\tau_\sh\mu)} \,\vec\Sigma_\varphi,
\\[1mm]
\vec\Gamma^\tau_{\sh,2}&= -\alpha\Big(\frac{\widehat{\rho}}{\rho}\,\frac{\pi}{2(\pi-\mu)}-\mathrm{K}_\mu\Big)\vec\varphi_0
+ \Big(\frac{\widehat{\rho}}{\rho}\,\frac{\tau_\pr\pi}
{2\tau_\sh(\tau_\pr\pi-\tau_\sh\mu)}-\frac{\mathrm{K}_{\mu,\tau}}{\tau_\sh}\Big)\vec\Sigma_\varphi,
\\[1mm]
\vec\Gamma^v_{\pr,1}&= \frac{1}{2(\pi-\mu)} \,\vec\varphi_0+\frac{\tau_\pr}{2(\tau_\pr\pi-\tau_\sh\mu)} \,\vec\Sigma_\varphi,
\\[1mm]
\vec\Gamma^v_{\pr,2}&= \Big(\frac{\widehat{\rho}}{\rho}\,\frac{1}{2(\pi-\mu)}+\mathrm{K}_\pi\Big)\vec\varphi_0+\tau_\pr\Big(\frac{\widehat{\rho}}{\rho}\,\frac{1}{3\tau_\pr\pi-4\tau_\sh\mu}+\mathrm{K}_{\pi,\tau}
\Big)\vec\Sigma_\varphi,
\\[1mm]
\vec\Gamma^\tau_{\pr,1}&= \frac{\alpha}{2(\pi-\mu)} \,\vec\varphi_0-\frac{1}{2(\tau_\pr\pi-\tau_\sh\mu)} \,\vec\Sigma_\varphi,
\\[1mm]
\vec\Gamma^\tau_{\pr,2}&= \alpha\Big(\frac{\widehat{\rho}}{\rho}\,\frac{1}{2(\pi-\mu)}+\mathrm{K}_\pi\Big)\vec\varphi_0
+\Big(\frac{\widehat{\rho}}{\rho}\,\frac{1}{2(\tau_\pr\pi-\tau_\sh\mu)}+\mathrm{K}_{\pi,\tau}
\Big)\vec\Sigma_\varphi.
\end{align*}
where
\begin{gather*}
\mathrm{K}_\mu = \frac{2\pi\mu\widetilde{\mu}-\widetilde{\mu}\pi^2-\widetilde{\pi}\mu^2}{2\mu(\pi-\mu)^2},\quad 
\mathrm{K}_{\mu,\tau}=\frac{2\tau_\pr\pi\tau_\sh\mu\widehat{\mu}-\widehat{\mu}\tau_\pr^2\pi^2-\widehat{\pi}\tau_\sh^2\mu^2}{2\tau_\sh\mu(\tau_\pr\pi-\tau_\sh\mu)^2},\\[1mm]
\mathrm{K}_\pi=\frac{\widetilde{\pi}-\widetilde{\mu}}{2(\pi-\mu)^2},\quad 
\mathrm{K}_{\pi,\tau}=\frac{\widehat{\pi}-\widehat{\mu}}{2(\tau_\pr\pi-\tau_\sh\mu)^2}.
\end{gather*}
\end{theorem}
\begin{proof}
We have
$$
\widetilde{C}(\mu,\pi)C(\widetilde{\mu}_1,\widetilde{\pi}_1)\vec\varepsilon(\V) = \frac{\widetilde{\mu}_1}{\mu} 
\vec\varepsilon(\V)  +\frac{\mu \widetilde{\pi}_1-\widetilde{\mu}_1\pi}{2\mu(\pi-\mu)}\,\Div\V\, 
\vec I
$$
and 
$$
C(\widetilde{\mu}_2,\widetilde{\pi}_2)\widetilde{C}(\mu,\pi)\vec \varphi_0 =\frac{\widetilde{\mu}_2}{\mu} 
\vec\varphi_0  +\frac{\mu \widetilde{\pi}_2-\widetilde{\mu}_2\pi}{2\mu(\pi-\mu)}\,\trace(\vec\varphi_0)\,\vec I
$$
so that
\begin{multline*}
\widetilde{C}(\mu,\pi) C(\widetilde{\mu}_1,\widetilde{\pi}_1)\vec\varepsilon(\vec v) : C(\widetilde{\mu}_2,\widetilde{\pi}_2)
\widetilde{C}(\mu,\pi)\vec\varphi_0=
\frac{\widetilde{\mu}_1\widetilde{\mu}_2}{\mu^2}\, \vec\varepsilon(\V): \vec\varphi_0 \\
+ \frac{\widetilde{\mu}_2(2\pi\mu\widetilde{\mu}_1-\widetilde{\mu}_1\pi^2-\widetilde{\pi}_1\mu^2)+\widetilde{\pi}_2
\mu^2(\widetilde{\pi}_1-\widetilde{\mu}_1)}{2\mu^2(\pi-\mu)^2}
\,\Div\V\, \trace(\vec\varphi_0) .
\end{multline*}
For the next steps see the proof of Theorem~\ref{FWI''_adjoint1}.
\end{proof}
~\\
\begin{theorem}[2D version of Theorem~\ref{FWI''_adjoint2}]~\\
Theorem~{\rm \ref{FWI''_adjoint2}} remains correct for the {\rm 2D} case when the definitions of the $\Upsilon$'s are replaced by
\begin{align*}
\vec\Upsilon^\rho_1 &= \Big(\frac{\widehat{\rho}}{\rho}+\frac{\widetilde{\mu}}{\mu}\Big)\vec\varphi_0
+ \Big(\frac{\widehat{\rho}}{\rho}+\frac{\widehat{\mu}}{\tau_\sh\mu}\Big)\vec\Sigma,\quad
\vec\Upsilon^\rho_2 = \frac{\widetilde{\pi}\mu-{\widetilde{\mu}}\pi}{2\mu(\pi-\mu)}\,\vec\varphi_0
+\frac{\widehat{\pi}\tau_\sh\mu-{\widehat{\mu}}\tau_\pr\pi}{2\tau_\sh\mu(\tau_\pr\pi-\tau_\sh\mu)}\,\vec\Sigma,\\[1mm]
\vec\Upsilon^v_{\sh,1}&= \Big(\frac{\widehat{\rho}}{\rho}+\frac{2\widetilde{\mu}}{\mu}\Big) \vec\varphi_0
+  \Big(\frac{\widehat{\rho}}{\rho}+\frac{2\widehat{\mu}}{\tau_\sh\mu}\Big)\vec\Sigma,\quad
\vec\Upsilon^v_{\sh,2}= \mathrm{K}_{\sh,\vec\varphi}\,\vec\varphi_0+\mathrm{K}_{\sh,\vec\Sigma}\,\vec\Sigma,\\[1mm]
\vec\Upsilon^\tau_{\sh,1}&= -\alpha\Big(\frac{\widehat{\rho}}{\rho}+\frac{2\widetilde{\mu}}{\mu}\Big) \vec\varphi_0
+  \Big(\frac{\widehat{\rho}}{\tau_\sh\rho}+\frac{2\widehat{\mu}}{\tau_\sh^2\mu}\Big)\vec\Sigma, \quad
\vec\Upsilon^\tau_{\sh,2}= -\alpha \mathrm{K}_{\sh,\vec\varphi}\, \vec\varphi_0+ \mathrm{K}_{\sh,\vec\Sigma}\, \vec\Sigma/\tau_\sh,\\[1mm]
\vec\Upsilon^v_{\pr}&=   \mathrm{K}_{\pr,\vec\varphi}\,\vec\varphi_0+\tau_\pr \, \mathrm{K}_{\pr,\vec\Sigma}\,\vec\Sigma,\quad
\vec\Upsilon^\tau_{\pr}= -\alpha\,  \mathrm{K}_{\pr,\vec\varphi}\,\vec\varphi_0+ \mathrm{K}_{\pr,\vec\Sigma}\,\vec\Sigma,
\end{align*}
where
\begin{align*}
\mathrm{K}_{\sh,\vec\varphi} &= \frac{2\pi\mu\widetilde{\mu}-\widetilde{\mu}\pi^2-\widetilde{\pi}\mu^2}{\mu(\pi-\mu)^2}-\frac{\widehat{\rho}}{\rho}\,\frac{\pi}{2(\pi-\mu)},\\[1mm]
 \mathrm{K}_{\sh,\vec\Sigma}& = \frac{2\tau_\pr\pi\tau_\sh\mu\widehat{\mu}-\widehat{\mu}\tau_\pr^2\pi^2-\widehat{\pi}\tau_\sh^2\mu^2}{\tau_\sh\mu(\tau_\pr\pi-\tau_\sh\mu)^2}-\frac{\widehat{\rho}}{\rho}\,\frac{\tau_\pr\pi}{2(\tau_\pr\pi-\tau_\sh\mu)},\\[1mm]
 \mathrm{K}_{\pr,\vec\varphi} &= \frac{\widehat{\rho}}{\rho}\,\frac{1}{2(\pi-\mu)}+
\frac{\widetilde{\pi}-\widetilde{\mu}}{(\pi-\mu)^2},\\[1mm]
 \mathrm{K}_{\pr,\vec\Sigma} &= \frac{\widehat{\rho}}{\rho}\,\frac{1}{2(\tau_\pr\pi-\tau_\sh\mu)}+
\frac{\widehat{\pi}-\widehat{\mu}}{(\tau_\pr\pi-\tau_\sh\mu)^2}.
\end{align*}
\end{theorem}
\begin{proof}
Let $S_0$ and $S_l$, $l=1,\dotsc,L$, be defined as in the proof of Theorem~\ref{FWI''_adjoint2}. Then, taking into account 
the formulas provided in the previous two proofs, 
\begin{align*}
S_0&= \widehat{\rho}_2\Bigg(\Big(\frac{\widehat{\rho}_1}{\rho^2}+\frac{\widetilde{\mu}_1}{\rho\mu}\Big) \vec\varepsilon(\V): \vec\varphi_0 +\frac{1}{\rho}\,\frac{\widetilde{\pi}_1\mu-\widetilde{\mu}_1\pi}{2\mu(\pi-\mu)} \,\Div\V\, \trace(\vec\varphi_0) \Bigg)\\
&\quad+\widetilde{\mu}_2 \Bigg(\Big(\frac{\widehat{\rho}_1}{\rho\mu}+\frac{2\widetilde{\mu}_1}{\mu^2}\Big) \vec\varepsilon(\V): \vec\varphi_0 +\Big(\frac{2\pi\mu\widetilde{\mu}_1-\widetilde{\mu}_1\pi^2-\widetilde{\pi}_1\mu^2}{\mu^2(\pi-\mu)^2}-\frac{\widehat{\rho}_1}{\rho}\,\frac{\pi}{2\mu(\pi-\mu)}\Big) \Div\V\, \trace(\vec\varphi_0) \Bigg)\\
&\quad+\widetilde{\pi}_2\Bigg(\frac{\widehat{\rho}_1}{\rho}\,\frac{1}{2(\pi-\mu)}+
\frac{\widetilde{\pi}_1-\widetilde{\mu}_1}{(\pi-\mu)^2}\Bigg)\Div\V\, \trace(\vec\varphi_0)
\end{align*}
and
\begin{align*}
S_l&= \widehat{\rho}_2\Bigg(\Big(\frac{\widehat{\rho}_1}{\rho^2}+\frac{\widehat{\mu}_1}{\rho\tau_\sh\mu}\Big) \vec\varepsilon(\V): \vec\varphi_l +\frac{1}{\rho}\,\frac{\widehat{\pi}_1\tau_\sh\mu-\widehat{\mu}_1\tau_\pr\pi}{2\tau_\sh\mu(\tau_\pr\pi-\tau_\sh\mu)} \,\Div\V\, \trace(\vec\varphi_l) \Bigg)\\
&\quad+\widehat{\mu}_2 \Bigg(\Big(\frac{\widehat{\rho}_1}{\rho\tau_\sh\mu}+\frac{2\widehat{\mu}_1}{\tau_\sh^2\mu^2}\Big) \vec\varepsilon(\V): \vec\varphi_l\\
&\qquad\qquad\quad +\Big(\frac{2\tau_\pr\pi\tau_\sh\mu\widehat{\mu}_1-\widehat{\mu}_1\tau_\pr^2\pi^2-\widehat{\pi}_1\tau_\sh^2\mu^2}{\tau_\sh^2\mu^2(\tau_\pr\pi-\tau_\sh\mu)^2}-\frac{\widehat{\rho}_1}{\rho}\,\frac{\tau_\pr\pi}{2\tau_\sh\mu(\tau_\pr\pi-\tau_\sh\mu)}\Big) \Div\V\, \trace(\vec\varphi_l) \Bigg)\\
&\quad+\widehat{\pi}_2\Bigg(\frac{\widehat{\rho}_1}{\rho}\,\frac{1}{2(\tau_\pr\pi-\tau_\sh\mu)}+
\frac{\widehat{\pi}_1-\widehat{\mu}_1}{(\tau_\pr\pi-\tau_\sh\mu)^2}\Bigg)\Div\V\, \trace(\vec\varphi_l).
\end{align*}
The next steps are as in the proof of Theorem~\ref{FWI''_adjoint2}.
\end{proof}

\end{document}